\documentclass[theo,computer]{manuart}
\usepackage[brouillon]{macros_liberte}

\DeclareMathOperator{\Br}{Br}
\DeclareMathOperator{\codim}{codim}
\DeclareMathOperator{\eff}{eff}
\newcommand{\et}{\textinmath{\'et}}
\newcommand{\ev}{\textinmath{\'ev}}
\DeclareMathOperator{\Det}{Det}
\DeclareMathOperator{\Fr}{Fr}
\DeclareMathOperator{\Gr}{Gr}
\DeclareMathOperator{\Hilb}{Hilb}
\DeclareMathOperator{\Hom}{Hom}
\DeclareMathOperator{\inv}{inv}
\DeclareMathOperator{\Mor}{Mor}
\DeclareMathOperator{\Pic}{Pic}
\DeclareMathOperator{\pr}{pr}
\DeclareMathOperator{\rg}{rg}
\DeclareMathOperator{\Spec}{Spec}
\DeclareMathOperator{\Sym}{Sym}
\DeclareMathOperator{\Tr}{Tr}
\DeclareMathOperator{\Val}{Val}
\DeclareMathOperator{\Vol}{Vol}
\newcommand{\card}{\sharp}
\newcommand{\llog}{\textinmath{\textup{\bf log}}}

\newcommand{\FF}{{\mathbf F}}
\newcommand{\NN}{{\mathbf N}}
\newcommand{\ZZ}{{\mathbf Z}}
\newcommand{\QQ}{{\mathbf Q}}
\newcommand{\RR}{{\mathbf R}}
\newcommand{\RRp}{\mathbf R_{\geq 0}}
\newcommand{\RRpp}{\mathbf R_{>0}}

\newcommand{\KK}{{\mathbf K}}
\newcommand{\LL}{{\mathbf L}}

\newcommand{\Adeles}{{\boldsymbol A}}
\newcommand{\ci}{{\mathit i}}

\newcommand{\dual}{^\vee}

\newcommand{\Aff}{{\mathbf A}}
\newcommand{\PP}{{\mathbf P}}

\newcommand{\exterieur}{{\mathsf{\Lambda}}}

\newcommand{\dega}{\mathop{\widehat\deg}}

\newcommand{\haar}[1]{{\text{d}#1\,}}
\newcommand{\oomega}{{\boldsymbol\omega}}
\newcommand{\ddelta}{{\boldsymbol\delta}}
\newcommand{\mmu}{\boldsymbol\mu}
\date{\today}
\title{Libert\'e et accumulation}
\author{Emmanuel Peyre}
\address{Institut Fourier\\
UFR IM$^2$AG\\
Universit\'e Grenoble Alpes et CNRS\\
BP 74\\ 38402 Saint-Martin d'H\`eres CEDEX\\ France}
\email{Emmanuel\!.Peyre@univ-grenoble-alpes.fr}
\urladdr{http://www-fourier.ujf-grenoble.fr/\~{}peyre}
\thanks{Je remercie l'universit\'e Grenoble Alpes, le CNRS,
  l'Institute for Advanced Studies de l'universit\'e de Bristol
  et le projet ANR GARDIO
  pour leur soutien durant la r\'edaction de cet article}
\makeindex
\makeglossary
\CDat
\begin{document}
\begin{abstract}
Le principe de Batyrev et Manin et ses variantes
donne une interpr\'etation conjecturale 
pr\'ecise pour le terme dominant du nom\-bre de points de hauteur 
born\'ee d'une vari\'et\'e alg\'ebrique dont l'oppos\'e du 
faisceau canonique est suffisamment positif.
Comme l'a clairement montr\'e le contre-exemple de Batyrev et Tschinkel
la mise en \oe uvre de ce principe n\'ecessite l'exclusion de domaines
d'accumulation qui sont le plus souvent d\'etermin\'es en proc\'edant par 
r\'ecurrence sur la dimension de la vari\'et\'e. Cette m\'ethode ne donne
cependant pas de crit\`ere direct permettant de dire si un point
rationnel donn\'e doit \^etre exclu ou pas.
L'ambition de cet article est de d\'efinir une mesure de la libert\'e d'un
point rationnel de sorte que les points d'une libert\'e suffisante
se r\'epartissent effectivement de mani\`ere uniforme sur
la vari\'et\'e, c'est-\`a-dire qu'ils soient distribu\'es
sur l'espace ad\'elique
associ\'e \`a la vari\'et\'e conform\'ement \`a la mesure de distribution
ad\'elique introduite dans un article ant\'erieur de l'auteur.
De ce point de vue, les points assez libres devraient \^etre ceux qui
respectent le principe de Batyrev et Manin.
\end{abstract}
\begin{altabstract}
The principle of Batyrev and Manin and its variants give a 
precise conjectural interpretation for the dominant term
for the number of points of bounded height on an algebraic variety
for which the opposite of the canonical line bundle is sufficiently positive.
As was clearly shown by the counter-example of Batyrev and Tschinkel,
the implementation of this principle requires the exclusion of accumulating
domains which, up to now, are found using an induction procedure on the
dimension of the variety. However this method does not yield a direct
caracterisation of the points to be excluded. The aim of this paper is to
propose a measure of freedom for rational points so that the points with 
sufficiently positive freedon are randomly distributed on the variety according
to a probability measure on the adelic space introduced by the author in
a previous paper. From that point of view the rational points which 
are sufficiently free ought be the ones which respect the principle of Batyrev
and Manin.
\end{altabstract}
\begin{deabstract}
Das Prinzip vom Batyrev und Manin und seine variationen geben eine genaue 
vermutliche Beschreibung des Hauptterms f\"ur die Zahl der rationalen Punkten
mit begrenzter H\"ohe auf eine Mannigfaltigkeit, deren kanonishe Geradeb\"undel
genug negativ ist. Wie das Gegenbeispiel von Batyrev und Tschinkel es erwiesen
hat, die Ausf\"uhrung dieses Prinzip benutzt die Ausschlie\ss ung von 
akkumulierenden Untermengen die, heutzutage, nur durch Induktion 
auf der Dimension herauszufinden sind. 
Das Ziel dieses Artikel ist ein Ma\ss\ f\"ur
die Freiheit rationalen Punkten vorzuschlagen, 
so da\ss\ die Punkten mit eine genugende 
Freiheit zuf\"allig nach einem Wahrscheinlichkeitsma\ss, die
in einem fr\"uheren Artikel des Autors definiert wurde, verteilt sind. 
Mit dieser Hinsicht,
die genugend freie Punkte sollten die Punkte sein, die das Prinzip von Batyrev
und Manin folgen. 
\end{deabstract}
\maketitle
\tableofcontents
\section{Introduction}%
\label{section.introduction}
Le programme d\'evelopp\'e par Y.~I.~Manin et ses collaborateurs
dans les ann\'ees 1980 (\cite{batyrevmanin:hauteur} et \cite{fmt:fano})
a permis une compr\'ehension en profondeur du comportement asymptotique
du nombre de points de hauteur born\'ee sur les vari\'et\'es alg\'ebriques
et en particulier sur les vari\'et\'es dont l'oppos\'e du
faisceau canonique est suffisamment positif, telles que
les vari\'et\'es de Fano.
\par
Un des points-clefs de ce programme est le concept de sous-vari\'et\'e
accumulatrice \cite[p.~30]{batyrevmanin:hauteur}, qu'on peut d\'ecrire,
au sens le plus faible, comme
une sous-vari\'et\'e stricte de la vari\'et\'e consid\'er\'ee
dont la contribution au nombre total de points sur la vari\'et\'e
n'est pas n\'egligeable. Lorsque la hauteur choisie n'est pas relative
\`a l'oppos\'e du faisceau canonique, la vari\'et\'e peut \^etre r\'eunion
de telles sous-vari\'et\'es, c'est ainsi le cas pour le produit de
deux droites projectives, cette situation est d\'ecrite plus
en d\'etail par V.~V.~Batyrev et Y.~Tschinkel dans
\cite{batyrevtschinkel:tamagawa}. Supposons maintenant que la hauteur
est relative \`a l'oppos\'e du fibr\'e canonique.
Le contre-exemple de V.~V. Batyrev et Y.~Tschinkel \`a la question
initiale de V.~V. Batyrev
et Y.~I. Manin \cite{batyrevtschinkel:counter} montre que,
m\^eme dans le cas d'une vari\'et\'e de Fano, la r\'eunion
de ces sous-vari\'et\'es peut, en g\'en\'eral, \^etre Zariski dense.
Diff\'erentes \'etudes bas\'ees sur la g\'eom\'etrie
des vari\'et\'es ont r\'ev\'el\'e que cette
situation est en fait tr\`es g\'en\'erale (\cite{lehmanntanimototschinkel}
et \cite{browningloughran:toomany}).
\par
L'esprit du programme de Manin est qu'en dehors de ces sous-vari\'et\'es
accumulatrices, le comportement asymptotique du nombre de points de hauteur 
born\'ee s'interpr\`ete en termes d'invariants globaux de la vari\'et\'e.
Dans \cite{rudulier:surfaces}, C.~Le Rudulier fait sauter un verrou
suppl\'ementaire en d\'emontrant que sur certains
espaces de Hilbert les sous-vari\'et\'es accumulatrices au sens pr\'ec\'edent
sont denses pour la topologie de Zariski, mais le nombre de points
de hauteur born\'ee sur le compl\'ementaire se comporte de la fa\c con
esp\'er\'ee. Comme nous le verrons plus loin, C.~Le Rudulier donne
\'egalement un exemple d'un ensemble mince faiblement accumulateur
qui ne s'obtient pas \`a partir d'une r\'eunion de sous-vari\'et\'es
faiblement accumulatrices.
\par
Dans~\cite[\S3]{peyre:fano} nous allions plus loin 
sur la question de la distribution asymptotique,
en construisant une mesure de probabilit\'e 
sur l'espace ad\'elique  correspondant \`a la vari\'et\'e, 
qui donnerait conjecturalement la distribution asymptotique 
des points de hauteur born\'ee en dehors des sous-vari\'et\'es 
accumulatrices. Notons au passage que pour cette mesure, les points
ad\'eliques d'une sous-vari\'et\'e stricte forment un ensemble de mesure
nulle, il en r\'esulte que cette distribution asymptotique ne saurait \^etre
valable que sur le compl\'ementaire des sous-vari\'et\'es accumulatrices
au sens pr\'ec\'edent.
\par
La m\'ethode d\'ecrite d\`es l'article de V.~V.~Batyrev et Y.~I.~Manin
pour trouver les sous-vari\'et\'es accumulatrices est de proc\'eder
par r\'eccurence sur la dimension des sous-vari\'et\'es. En effet,
si l'on note~$B$ la borne choisie pour la hauteur, le comportement
asymptotique attendu sur une vari\'et\'e~$V$
en dehors des sous-vari\'et\'es accumulatrices
est de la forme $CB^a\log(V)^{b-1}$, o\`u~$a$ et~$b$ ont des interpr\'etations
g\'eom\'etriques. \`A partir de ces invariants g\'eom\'etriques, on
peut donc d\'efinir une notion de sous-vari\'et\'es \emph{g\'eom\'etriquement}
accumulatrices: ce sont celles pour lequel la pr\'ediction pour le
nombre de points de la sous-vari\'et\'e est
au moins du m\^eme ordre de grandeur que
celui pr\'edit pour la vari\'et\'e dans son ensemble.
Cette approche s'est r\'evel\'e tr\`es efficace dans la plupart
des cas connus jusqu'\`a aujourd'hui.
\par
Toutefois, elle a l'inconv\'enient de ne pas donner de crit\`ere permettant
de d\'eterminer directement si un point rationnel donn\'e doit \^etre exclu
du d\'enombrement. Il est alors enrichissant de comparer cette
situation avec l'analogue g\'eom\'etrique. Dans l'\'etude des espaces
de morphismes d'une courbe dans une vari\'et\'e donn\'ee, il convient de
consid\'erer les morphismes tr\`es libres dont les d\'eformations 
permettent de couvrir la vari\'et\'e. Ces morphismes sont ais\'ement
caract\'eris\'es \`a l'aide des pentes de l'image inverse du fibr\'e
tangent sur la courbe. Or la g\'eom\'etrie d'Arakelov et notamment
la th\'eorie des pentes d\'evelopp\'ee par J.-B.~Bost fournissent un analogue
arithm\'etique \`a ces pentes. Toutefois, comme toujours dans
cette traduction, l'invariant obtenu au lieu d'\^etre entier
est r\'eel et on peut donc \^etre confront\'e \`a des situations
o\`u la pente minimale d'une famille de points rationnels tend vers~$0$.
\par
L'ambition de ce texte est donc d'utiliser les pentes de la g\'eom\'etrie
d'Arakelov pour d\'efinir un invariant qui mesure la \emph{libert\'e}
d'un point rationnel. L'invariant que nous construisons est un \'el\'ement de
l'interval r\'eel $[0,1]$. Il fait intervenir le choix
d'une m\'etrique ad\'elique dont on munit la vari\'et\'e, m\'etrique qui
est l'analogue dans notre cadre ad\'elique
d'une m\'etrique riemannienne. Dans les
exemples \'etudi\'es, nous montrons que les points qu'il convient 
d'exclure ont une libert\'e qui tend vers $0$ lorsque leur hauteur
tend vers~$+\infty$, ce qui permet d'\'enoncer une variante de la conjecture
raffin\'ee de Batyrev et Manin qui ne semble pas en contradiction avec les
exemples consid\'er\'es ici.
\par
Certes, il convient de relativiser les progr\`es que permettent
ce nouvel invariant en raison de la difficult\'e de sa d\'etermination
en g\'en\'eral; n\'eanmoins c'est un invariant tr\`es naturel
et les quelques indices qui suivent confirment sa capacit\'e \`a
d\'etecter une obstruction \`a la distribution uniforme des points
rationnels.
\par
Ce texte est organis\'e de la fa\c con  suivante: on commence
par fixer le cadre g\'eom\'etrique, avant de pr\'eciser la notion
de m\'etrique ad\'elique.  Le paragraphe suivant consiste en des rappels sur
les notions de pentes, ce qui nous permet de d\'efinir l'invariant 
qui est l'objet de cet article.
Nous en donnons ensuite quelques propri\'et\'es
\'el\'ementaires 
avant d'introduire une nouvelle variante du principe de Batyrev et Manin, la
formule empirique \eqref{equ.laformule}.
Le reste de l'article est consacr\'e \`a l'application
de la notion de libert\'e \`a divers
exemples ou contre-exemples connus: l'espace
projectif, le produit de vari\'et\'es, les quadriques,
le contre-exemple de V.~V. Batyrev et Y. Tschinkel ainsi
que ceux de C.~Le Rudulier.

Je remercie T.~Browning, \'E.~Gaudron
et D.~Loughran pour plusieures discussions
et suggestions qui m'ont permis d'am\'eliorer ce texte.
\section{Cadre}%
\label{section.cadre}
On fixe un corps de nombres~$\KK$.
Pour toute $\KK$-alg\`ebre commutative~$A$ et toute vari\'et\'e~$V$
sur~$\KK$, on note $V_A$ le produit sch\'ematique $V\times_{\Spec(\KK)}\Spec(A)$ et
$V(A)$ l'ensemble $\Mor_{\Spec(\KK)}(\Spec(A),V)$ des $A$-points de $V$.
\par
On note $\Val(\KK)$ l'ensemble des places du corps de nombres~$\KK$.
Soit~$w$ une place de $\KK$. On note $\KK_w$ le compl\'et\'e de~$\KK$ pour~$w$.
Soit~$v$ la place de~$\QQ$ d\'eduite de~$w$ par restriction, on pose
\[|x|_w=|N_{\KK_w/\QQ_v}(x)|_v\]
pour $x\in \KK_w$. Si $\KK_w$ est isomorphe au corps des complexes,
$|\cdot|_w$ correspond au carr\'e du module, sinon 
l'application $|\cdot|_w$ est une valeur absolue 
repr\'esentant~$w$. Cette normalisation permet d'\'ecrire la formule
du produit pour~$\KK$ sous la forme
\begin{equation}
\label{equ.product}
\prod_{w\in\Val(\KK)}|x|_w=1
\end{equation}
pour $x\in \KK^\times$.
\par
\begin{conv}
Dans la suite, on ne consid\`ere que de \emph{belles}%
\index{Variete=Vari\'et\'e>belle}%
\index{Bonne vari\'et\'e}
vari\'et\'es sur~$\KK$, c'est-\`a-dire des vari\'et\'es projectives
lisses et g\'eom\'etriquement int\`egres sur le corps de nombres~$\KK$.
\end{conv}
\section{M\'etriques ad\'eliques}%
\label{section.metrique}

\subsection{Normes $w$-adiques}%
\label{subsection.normes}
Nous introduisons maintenant une notion de \emph{nor\-me classique}
sur les espaces vectoriels sur les compl\'et\'es du corps de nombres~$\KK$.
\begin{defi}
Soit~$w$ une place de~$\KK$.
Soit~$E$ un espace vectoriel sur le corps compl\'et\'e~$\KK_w$.
Dans ce texte, on appelle \emph{norme $w$-adique}%
\index{Norme>wadique=$w$-adique}%
\index{wadique=$w$-adique>norme=(Norme ---)}
une application $\Vert\cdot\Vert_w:x\mapsto\Vert x\Vert_w$
de~$E$ dans~$\RR$ qui v\'erifie les conditions suivantes
\begin{conditions}
\item On a la relation $\Vert\lambda x\Vert_w=|\lambda|_w\,\Vert x\Vert_w$
pour $\lambda\in \KK_w$ et $x\in E$;
\item Si~$w$ est une place ultram\'etrique, alors
on a l'in\'egalit\'e $\Vert x+y\Vert_w\leq\max(\Vert x\Vert_w,\Vert y\Vert_w)$
pour $x,y\in E$;
\item Si $w$ est une place r\'eelle,
alors $\Vert x+y\Vert_w\leq\Vert x\Vert_w+\Vert y\Vert_w$;
\item Enfin, si $w$ est complexe,
alors $\Vert x+y\Vert_w^{1/2}\leq\Vert x\Vert_w^{1/2}+\Vert y\Vert_w^{1/2}$;
\end{conditions}
Nous dirons que cette norme est \emph{classique}%
\index{Norme>classique}%
\index{Classique (Norme ---)}
si elle v\'erifie en outre les conditions suivantes:
\begin{conditions}
\setcounter{enumi}{1}\renewcommand{\theenumi}{\roman{enumi}$'$}
\item Si $w$ est ultram\'etrique,
l'image de $\Vert\cdot\Vert_w$ est contenue dans l'image
de la valeur absolue $|\cdot|_w$;
\item Si $w$ est r\'eelle, il existe une forme quadratique d\'efinie
positive~$q$
sur~$E$ telle que $\Vert x\Vert_w=\sqrt{q(x)}$;
\item Si $w$ est complexe, il existe une forme hermitienne
d\'efinie positive~$h$ sur~$E$ telle que
$\Vert x\Vert_w=h(x)$.
\end{conditions}
\end{defi}
\begin{rema}
\label{rema.II.3.11}
On notera que pour une place ultram\'etrique~$w$, la donn\'ee
d'une norme classique $\Vert\cdot\Vert_w$ sur un $\KK_w$-espace
vectoriel~$E$ est \'equivalente \`a la donn\'ee
du $\mathcal O_w$-module
\[\mathcal E=\{\,x\in E\mid \Vert x\Vert_w\leq 1\,\},\]
qui est libre de rang la dimension de $E$.
En outre, pour tout $x\in E$, on a dans ce cas l'\'egalit\'e
\[\Vert x\Vert_w=\min\{|\lambda|_w^{-1},\ 
\lambda\in\KK_w\setminus\{0\}\text{ tels que }\lambda x\in\mathcal E\}.\]
\end{rema}

\subsection{Normes ad\'eliques}
Dans ce paragraphe, la lettre~$V$ d\'esigne une \emph{belle}
vari\'et\'e sur~$\KK$.
Du point de vue de la g\'eom\'etrie d'Arakelov, une hauteur
va \^etre d\'efinie \`a partir d'un fibr\'e en droites
muni de normes.
\begin{defi}
Soit~$E$ un fibr\'e vectoriel sur~$V$. On note ${\pi:E\to V}$ son
morphisme structural. Pour toute extension~$\LL$ de~$\KK$
et tout $x\in V(\LL)$, on note~$E(x)$ le $\LL$-espace vectoriel
$\pi^{-1}(x)$.
Soit $w\in\Val(\KK)$. Une \emph{norme $w$-adique}%
\index{Norme>wadique=$w$-adique}
sur~$E$
est une application $\Vert\cdot\Vert_w:x\mapsto \Vert x\Vert_w$%
\glossary{$\Vert\cdot\Vert_w$} de
$E(\KK_w)$ dans $\RRp$ v\'erifiant les conditions suivantes:
\begin{conditions}
\item Pour tout $x\in V(\KK_w)$,
la restriction de $\Vert\cdot\Vert_w$ \`a $E(x)$ est une norme
$w$-adique.
\item Pour tout ouvert~$U$ de~$V$ et toute section~$s$
de~$E$ sur~$U$, l'application de $U(\KK_w)$ dans $\RR$
donn\'ee par $x\mapsto \Vert s(x)\Vert_w$ est continue
pour la topologie $w$-adique.
\end{conditions}
Cette norme $\Vert\cdot\Vert_w$ est dite \emph{classique}%
\index{Norme>wadique=$w$-adique>classique}
si sa restriction \`a $E(x)$ est classique pour tout 
$x\in V(\KK_w)$.
\end{defi}
\begin{exem}
On suppose donn\'e un \emph{mod\`ele}
\index{Modele=Mod\`ele}
de $V$ sur $\mathcal O_w$, c'est-\`a-dire
un sch\'ema $\mathcal V$ projectif et
lisse sur $\mathcal O_w$ muni d'un isomorphisme de $\mathcal V_{\KK_w}$ 
sur $V_{\KK_w}$ et un mod\`ele de $E$
au-dessus de $\mathcal V$, c'est-\`a-dire un fibr\'e vectoriel
$\mathcal E$ au-dessus de $\mathcal V$ muni d'un isomorphisme
de fibr\'es vectoriels de $\mathcal E_{\KK_w}$ sur $E_{\KK_w}$. 
Alors, comme $\mathcal V$ est projective, l'application
naturelle de $\mathcal V(\mathcal O_w)$ dans $V(\KK_w)$ est bijective.
Soit $x\in V(\KK_w)$. Soit $\tilde x$
le point correspondant de $\mathcal V(\mathcal O_w)$.
Alors $\tilde x^*(\mathcal E)$ d\'efinit un sous-$\mathcal O_w$-module
de $E(x)$, que l'on notera $\mathcal E(x)$,
qui est libre de rang la dimension de $E(x)$.
Par la remarque~\ref{rema.II.3.11}, cela d\'efinit une norme
classique sur $E(x)$. On obtient ainsi une norme $w$-adique
classique sur $E$ qu'on dira \emph{associ\'ee au mod\`ele} $\mathcal E$%
\index{Norme>wadique=$w$-adique>associee=associ\'ee \`a un mod\`ele}
de~$E$.
\end{exem}
\begin{defi}\label{defi.II.3.14}
Soit~$E$ un fibr\'e vectoriel au-dessus de~$V$.
Une \emph{norme ad\'elique}%
\index{Norme>adelique=ad\'elique}
sur $E$ est une famille
$(\Vert\cdot\Vert_w)_{w\in\Val(\KK)}$, o\`u
$\Vert\cdot\Vert_w$ est une norme $w$-adique
sur $E$ pour tout $w\in\Val(\KK)$,
de sorte que la condition suivante
soit v\'erifi\'ee: il existe un ensemble fini de places~$S$ et
un mod\`ele~$\mathcal E$ de~$E$ sur l'anneau $\mathcal O_S$ des $S$-entiers tels
que pour toute place $w\in\Val(\KK)\setminus S$, la norme $\Vert\cdot\Vert_w$
soit la norme associ\'ee \`a $\mathcal E_{\mathcal O_w}$.
On dira que cette norme ad\'elique est \emph{classique} si toutes les normes 
$\Vert\cdot\Vert_w$ sont classiques.
\par
Un \emph{fibr\'e vectoriel (\resp\ fibr\'e en droites)
ad\'eliquement norm\'e}%
\index{Fibrevectoriel=Fibr\'e
vectoriel>adeliquementnorme=ad\'eliquement norm\'e}
sur~$V$ est la donn\'ee d'un fibr\'e
vectoriel (\resp\ un fibr\'e en droites)~$E$ au-dessus de $V$, muni d'une 
norme ad\'elique.
Un \emph{fibr\'e ad\'elique}%
\index{Fibr\'e ad\'elique}
d\'esigne, dans ce texte, un fibr\'e ad\'eliquement norm\'e
dont la norme est classique.
\par
Dans ce texte, nous r\'eserverons la terminologie 
\og\emph{m\'etrique ad\'elique
sur $V$}\fg%
\index{Metriqueadelique=M\'etrique ad\'elique}%
\index{Adelique=Ad\'elique>Metrique=M\'etrique (---)}\footnote{Dans des
textes ant\'erieurs, j'ai appel\'e m\'etrique ad\'elique
ce que j'appelle ici plus justement norme ad\'elique 
pour \'eviter des confusions.}
\`a la donn\'ee d'un norme ad\'elique classique sur le fibr\'e tangent $T_V$
de~$V$. 
\end{defi}
\begin{listexems}
\exemple\label{exem.4.4.a}
Le fibr\'e en droites trivial $V\times \Aff^1_\KK$, muni
des normes naturelles ${\Vert\cdot\Vert_w}$
d\'efinies par $(x,\lambda)\mapsto |\lambda|_w$
pour $w\in\Val(\KK)$ est un fibr\'e vectoriel ad\'eliquement norm\'e
qu'on dira \'egalement trivial.
Consid\'erons l'ensemble $\mathcal F$ des familles $(f_w)_{w\in\KK}$,
o\`u $f_w$ est une application continue de $V(\KK_w)$ dans $\RRpp$
pour tout $w\in\Val(\KK)$, et o\`u $f_w$ est
constante de valeur de $1$ en dehors 
d'un ensemble fini de places.
Alors l'application qui \`a un \'el\'ement $(f_w)_{w\in\KK}$
de $\mathcal F$ associe la famille $(f_w\Vert\cdot\Vert_w)_{w\in\Val(\KK)}$
est une bijection de $\mathcal F$ sur les normes ad\'eliques sur le 
fibr\'e en droites trivial.
\exemple\label{exem.4.4.b}
Plus g\'en\'eralement, soit~$L$ un fibr\'e en droites
et soit $(\Vert\cdot\Vert_w)_{w\in\Val(\KK)}$ une norme ad\'elique sur~$L$.
Toute norme ad\'elique sur le fibr\'e $L$ est de la forme
$(f_w\Vert\cdot\Vert_w)_{w\in\Val(\KK)}$ pour une unique famille $(f_w)_{w\in\KK}$
de $\mathcal F$.
Cela d\'ecoule du fait que si $\mathcal E$ et $\mathcal E'$
sont des mod\`eles d'un fibr\'e vectoriel~$E$ sur un anneau de
$S$-entiers $\mathcal O_S$, alors il existe un ensemble fini
de places $S'$ contenant~$S$ de sorte que l'isomorphisme
de $\mathcal E_\KK$ sur $\mathcal E_\KK'$ provienne d'un isomorphisme de
$\mathcal E_{\mathcal O_{S'}}$ sur $\mathcal E'_{\mathcal O_{S'}}$.
\exemple\label{exem.fibre.somme}
Soient~$E$ (\resp~$E'$) un fibr\'e vectoriel sur
la vari\'et\'e~$V$ muni
d'une norme ad\'elique
classique ${(\Vert\cdot\Vert_w)_{w\in\Val(\KK)}}$ 
(\resp ${(\Vert\cdot\Vert'_w)_{w\in\Val(\KK)}}$). 
Alors on d\'efinit une norme ad\'elique classique 
$(\Vert\cdot\Vert''_w)_{w\in\Val(\KK)}$
sur le fibr\'e vectoriel $E\oplus E'$ par les conditions suivantes:
\begin{conditions}
\item Si $w\in \Val(\KK)$ est une place non archim\'edienne, alors on pose
$\Vert (y,y')\Vert''_w=\max(\Vert y\Vert_w,\Vert y'\Vert'_w)$
pour tout $x\in V(\KK_w)$, tout $y\in E(x)$ et tout $y'\in E'(x)$.
\item Si $w$ est une place r\'eelle de $\KK$, alors on pose
$\Vert (y,y')\Vert''_w=(\Vert y\Vert^2_w+{\Vert y'\Vert'_w}^2)^{1/2}$
pour tout $x\in V(\KK_w)$, tout $y\in E(x)$ et tout $y'\in E'(x)$.
\item Si $w$ est une place complexe, alors on d\'efinit
$\Vert (y,y')\Vert''_w=\Vert y\Vert_w+\Vert y'\Vert'_w$
pour tout $x\in V(\KK_w)$, tout $y\in E(x)$ et tout $y'\in E'(x)$.
\end{conditions}
Le fibr\'e vectoriel ad\'eliquement norm\'e ainsi obtenu
sera appel\'e la \emph{somme directe} des fibr\'es vectoriels ad\'eliquement
norm\'es $E$ et $E'$.%
\index{Fibrevectoriel=Fibr\'e
vectoriel>adeliquementnorme=ad\'eliquement norm\'e>somme directe}
\exemple\label{exem.fibre.produit}
Soient~$E$ (\resp~$E'$) un fibr\'e vectoriel sur
la vari\'et\'e~$V$ muni
d'une norme ad\'elique
classique ${(\Vert\cdot\Vert_w)_{w\in\Val(\KK)}}$ 
(\resp ${(\Vert\cdot\Vert'_w)_{w\in\Val(\KK)}}$). 
Alors il existe une unique norme ad\'elique classique 
$(\Vert\cdot\Vert''_w)_{w\in\Val(\KK)}$
sur le fibr\'e vectoriel $E\otimes E'$ qui v\'erifie:
\[\Vert y\otimes y'\Vert''_w=\Vert y\Vert_w\Vert y'\Vert'_w\]
pour $w\in\Val(\KK)$, $x\in V(\KK_w)$, $y\in E(x)$ et $y'\in E'(x)$.
Le fibr\'e vectoriel ad\'eliquement norm\'e ainsi obtenu
sera appel\'e la \emph{produit tensoriel} des fibr\'es vectoriels ad\'eliquement
norm\'es $E$ et $E'$.%
\index{Fibrevectoriel=Fibr\'e
vectoriel>adeliquementnorme=ad\'eliquement norm\'e>produit tensoriel}
\exemple\label{exem.4.4.e}
Soit~$E$ un fibr\'e vectoriel sur~$V$ muni
d'une norme ad\'elique que l'on note ${(\Vert\cdot\Vert_w)_{w\in\Val(\KK)}}$.
Soit~$F$ un sous-fibr\'e de~$E$.
La famille de normes
$(\Vert\cdot\Vert'_w)_{w\in\Val(\KK)}$
sur le fibr\'e vectoriel $F$ d\'efinie par les relations
\[\Vert y\Vert'_w=\Vert y\Vert_w\]
pour $w\in\Val(\KK)$, $x\in V(\KK_w)$ et $y\in F(x)$
est une norme ad\'elique
appel\'ee la \emph{restriction} de la norme ad\'elique de~$E$
\`a~$F$.%
\index{Fibrevectoriel=Fibr\'e
vectoriel>adeliquementnorme=ad\'eliquement norm\'e>restriction}
Si la norme ad\'elique de~$E$ est classique, sa restriction
\`a~$F$ l'est \'egalement.

\exemple\label{exem.4.4.f}
On conserve les notations de l'exemple pr\'ec\'edent.
Soit $w\in\Val(\KK)$, on d\'efinit une application $E/F(\KK_w)\to\RR_+$
par
\[\Vert \overline y\Vert''_w=\inf_{y\in\overline y}\Vert y\Vert_w\]
pour $w\in\Val(\KK)$, $x\in V(\KK_w)$ et $\overline y\in E(x)/F(x)$.
Mais comme cette borne inf\'erieure est atteinte, $\Vert\cdot\Vert''_w$
est une norme sur le fibr\'e $E/F$.
En outre, si $\Vert\cdot\Vert_w$ est d\'efini par un mod\`ele
$\mathcal E$ et la restriction $\Vert\cdot\Vert_w'$ par
un sous-fibr\'e $\mathcal F$, alors la norme $\Vert\cdot\Vert''_w$
est d\'efinie par le quotient $\mathcal E/\mathcal F$.

La norme ad\'elique 
$(\Vert\cdot\Vert''_w)_{w\in\Val(\KK)}$
sur le fibr\'e vectoriel $E/F$
est appel\'ee \emph{norme d\'eduite de la norme ad\'elique de~$E$
par passage au quotient}.%
\index{Fibrevectoriel=Fibr\'e
vectoriel>adeliquementnorme=ad\'eliquement norm\'e>passage au quotient}
Supposons que la norme ad\'elique de~$E$ soit classique. La norme qui s'en 
d\'eduit par passage au quotient
l'est \'egalement. De mani\`ere plus pr\'ecise, soit~$w$
une place de~$\KK$ et soit~$x$ un \'el\'ement de $V(\KK_w)$.
Si $w$ est une place finie, alors le sous-$\mathcal O_w$-module 
de $E(x)/F(x)$
associ\'e \`a $\Vert\cdot\Vert''_w$ est le quotient 
du $\mathcal O_w$-module associ\'e \`a $\Vert\cdot\Vert_w$
par son intersection avec l'intersection avec $F(x)$;
si $w$ est r\'eelle (\resp\ complexe) alors la projection
d\'efinit un isomorphisme d'espace
quadratique (\resp\ hermitien) de l'orthogonal de $F(x)$
sur $E/F(x)$.

\exemple\label{exem.4.4.g}
Soit~$E$ un fibr\'e ad\'elique et~$n$
un entier positif. On d\'efinit
une norme ad\'elique classique sur le fibr\'e $\exterieur^nE$ de fa\c con
\`a obtenir une structure de $\lambda$-anneau sur l'anneau
de Grothendieck des fibr\'es vectoriels munis de normes ad\'eliques
classiques \cite[\S2]{berthelot:lambda}.
En particulier, si $E$ est une somme directe $F\oplus F'$,
les normes choisies doivent \^etre compatibles avec
l'isomorphisme canonique de $\exterieur^nE$ sur
\[\bigoplus_{p+q=n}\exterieur^pF\otimes\exterieur^qF'.\]
Soit $w$ une place de $K$ et $x\in V(\KK_w)$
si $w$ est non-archim\'edienne, la restriction de $\Vert\cdot\Vert_w$
\`a $\exterieur^nE(x)$ est donn\'e par le $\mathcal O_w$-module engendr\'e
par les produits $y_1\wedge\dots\wedge y_n$, o\`u $y_i\in E(x)$ avec
$\Vert y_i\Vert_w\leq 1$ pour $i\in\{1,\dots,n\}$.
Si $w$ est une place archim\'edienne, alors on prend
la norme associ\'ee \`a la forme bilin\'eaire caract\'eris\'ee
par la formule
\[\langle x_1\wedge\dots\wedge x_n,y_1\wedge\dots\wedge y_n
\rangle_{\exterieur^n E}
=\det(\langle x_i,y_i\rangle_E),\]
o\`u $x_1,\dots,x_n,y_1,\dots,y_n\in E(x)$ et $\langle\cdot,\cdot\rangle_E$
d\'esigne la forme d\'efinissant la norme sur $E(x)$.
En particulier, si $(e_1,\dots,e_r)$ est une base orthonorm\'ee
de $E(x)$, alors les produits $e_{i_1}\wedge\dots\wedge e_{i_n}$
avec $1\leq i_1<i_2<\dots<i_n\leq r$ forment une base orthonorm\'ee
de $\exterieur^nE(x)$.
En prenant pour~$n$ le rang de $E$, on munit ainsi le fibr\'e 
$\det(E)=\exterieur^nE$
d'une structure de fibr\'e en droite ad\'eliquement norm\'e.
En particulier, une m\'etrique ad\'elique sur~$V$ permet de d\'efinir
une norme ad\'elique classique sur
le fibr\'e anticanonique $\omega_V^{-1}=\det(T_V)$.

\exemple\label{exem.4.4.h}
Soit~$X$ une belle vari\'et\'e sur~$K$ et soit $f:X\to Y$
un morphisme de vari\'et\'es. Soit~$E$ un fibr\'e vectoriel 
au-dessus de~$Y$ muni d'une norme ad\'elique 
$(\Vert\cdot\Vert_w)_{w\in\Val(\KK)}$.
Soit $x\in X(\KK)$. On peut identifier $f^*(E)(x)$ avec
$E(f(x))$. Pour tout place~$w$ de~$\KK$, la norme $\Vert\cdot\Vert_w$
sur $E(f(x))$ induit donc une norme sur $f^*(E)(x)$. Le fibr\'e
vectoriel $f^*(E)$ muni de la norme ad\'elique ainsi d\'efinie
est appel\'e l'\emph{image inverse par~$f$} du fibr\'e vectoriel
ad\'eliquement norm\'e~$E$.%
\index{Imageinverse=Image inverse>dunfibreadeliquementnorme=d'un fibr\'e 
ad\'eliquement norm\'e}%
\index{Fibrevectoriel=Fibr\'e
vectoriel>adeliquementnorme=ad\'eliquement norm\'e>image inverse}
\end{listexems}

\begin{defi}\label{defi.II.3.16}
Soient~$E$ et~$F$ des fibr\'es ad\'eliquement norm\'es
au-dessus de~$V$. On note $(\Vert\cdot\Vert_w)_{w\in\Val(\KK)}$ 
(\resp\ $(\Vert\cdot\Vert'_w)_{w\in\Val(\KK)}$) la norme ad\'elique
sur~$E$ (\resp\ $F$).
\par
Un \emph{morphisme de~$E$ dans~$F$} est
un morphisme~$\varphi$ de fibr\'es
vectoriels de~$E$ dans~$F$.
\par
Un \emph{plongement de~$E$ dans~$F$} est un morphisme~$\varphi$ de fibr\'es
vectoriels de~$E$ dans~$F$ tel que pour toute place~$w$
de~$\KK$, tout point~$x$ appartenant \`a~$V(\KK_w)$ et tout
$y\in E(x)$, on ait $\Vert\varphi(y)\Vert'_w=\Vert y\Vert_w$.%
\index{Fibrevectoriel=Fibr\'e
vectoriel>adeliquementnorme=ad\'eliquement norm\'e>plongement}%
\index{Plongement>defibresvectorielsadeliquementnormes=de
  fibr\'es vectoriels ad\'eliquement norm\'es}
\par
Les fibr\'es vectoriels ad\'eliquement norm\'es~$E$ et~$F$
seront dits \emph{\'equivalents}%
\index{Fibrevectoriel=Fibr\'e
vectoriel>adeliquementnorme=ad\'eliquement 
norm\'e>equivalence=\'equivalence}
si et seulement s'il existe
un isomorphisme de fibr\'es vectoriels $\varphi:E\to F$
et une famille $(\lambda_w)_{w\in\Val(\KK)}$ de nombres r\'eels
tels que
\begin{conditions}
\item L'ensemble $\{\,w\in\Val(\KK)\mid \lambda_w\neq 1\,\}$
est fini;
\item Le produit $\prod_{w\in\Val(\KK)}\lambda_w$ vaut $1$;
\item On a la relation $\Vert\varphi(y)\Vert'_w=\lambda_w\Vert y\Vert_w$
pour tout $w\in\Val(\KK)$, tout $x\in V(\KK_w)$ et tout $y\in E(x)$.
\end{conditions}
L'ensemble des classes d'\'equivalence de fibr\'es
en droites munis d'une norme ad\'elique pour la relation d'\'equivalence
pr\'ec\'edente forme un groupe pour le produit tensoriel,
qu'on note $\mathcal H(V)$.%
\glossary{$\mathcal H(V)$}
\end{defi}

\begin{exem}
Dans le cas o\`u $V=\Spec(\KK)$, la donn\'ee d'un fibr\'e
ad\'elique sur~$V$ est \'equivalente \`a celle d'un
$\KK$-espace vectoriel de dimension finie~$n$
muni d'une famille de normes classiques
$(\Vert\cdot\Vert_w)_{w\in\Val(\KK)}$
de sorte que pour toute base $(e_1,\dots,e_n)$ de~$E$, il existe
un ensemble fini de places $S$ tel que
\[\Bigl\Vert \sum_{i=1}^nx_ie_i\Bigr\Vert_w=\max_{1\leq i\leq n}(|x_i|_w)\]
pour toute place~$w$ en dehors de~$S$ et
tout $(x_1,\dots,x_n)\in \KK_w^n$.
L'ensemble~$M$ des \'el\'ements~$x$ de~$E$ tels que
$\Vert x\Vert_w\leq 1$ pour toute place finie~$w$ de~$K$
est alors un $\mathcal O_\KK$-module projectif de rang constant~$n$.
Notons qu'inversement $\Vert\cdot\Vert_w$ est la norme $w$-adique
d\'efinie par $M\otimes_{\mathcal O_K}\mathcal O_w$, si bien que
la donn\'ee d'un fibr\'e ad\'elique sur~$\Spec(\KK)$ est \'egalement
\'equivalente \`a celle d'un \emph{r\'eseau ad\'elique sur~$\KK$}, 
c'est-\`a-dire
d'un $\mathcal O_\KK$-module projectif~$M$
de rang fini constant muni, pour tout place 
archim\'edienne $w$ d'une norme classique $\Vert\cdot\Vert_w$ sur
$M_w=M\otimes_{\mathcal O_K}\KK_w$.

En particulier, un fibr\'e en droites sur~$\Spec(\KK)$
muni d'une norme ad\'elique ad\'elique est
d\'efini par un $\KK$-espace vectoriel~$E$ de dimension un
muni d'une famille de normes $(\Vert\cdot\Vert_w)_{w\in\Val(\KK)}$
de sorte que pour tout \'el\'ement non nul~$e$ de~$E$
l'ensemble de places  ${\{w\in\Val(\KK)\mid\,\Vert e\Vert_w\neq 1\}}$
soit fini. Par la formule du produit~\eqref{equ.product}, le produit 
$\prod_{w\in\Val(\KK)}\Vert e\Vert_w$ est ind\'ependant de l'\'el\'ement
non nul~$e$ de~$E$ et on peut d\'efinir le \emph{degr\'e} de $E$
par la formule
\begin{equation}
  \label{equa.pente.droite}
  \dega(E)=-\log\Biggl(\prod_{w\in\Val(\KK)}\Vert e\Vert_w\Biggr).
\end{equation}
\glossary{$\dega(E)$}%
\index{Degre=Degr\'e>adelique=ad\'elique}%
Supposons que le degr\'e de~$E$ soit \'egal \`a~$0$. Soit~$e$
une \'el\'ement non nul de~$E$ et posons 
$\lambda_w=\Vert e\Vert_w^{-1}$ pour $w\in\Val(\KK)$. Alors l'application
lin\'eaire $\varphi: \KK\to E$ qui envoie $1$ sur $e$ d\'efinit un
isomorphisme de fibr\'e ad\'elique du fibr\'e ad\'elique trivial
(exemple~\ref{exem.4.4.a}) sur~$E$ muni de la norme ad\'elique
$(\lambda_w\Vert\cdot\Vert_w)_{w\in\Val(\KK)}$. Par cons\'equent,
le degr\'e d\'efinit un isomorphisme de groupes
de $\mathcal H(\Spec(\KK))$ sur~$\RR$.
\end{exem}
\begin{prop}
  \label{prop.maj.morph}
  Soient~$E$ et~$F$ des fibr\'es vectoriels ad\'eliquement norm\'es
  au-dessus de la vari\'et\'e~$V$
  et soit~$\varphi$ un morphisme de~$E$ dans~$F$.
  On note ${(\Vert\cdot\Vert_w)_{w\in\Val(\KK)}}$ 
  (\resp\ ${(\Vert\cdot\Vert'_w)_{w\in\Val(\KK)}}$) la norme ad\'elique
  sur~$E$ (\resp\ $F$).  Alors il existe une
  famille $(\lambda_w)_{w\in\Val(\KK)}$ de nombres r\'eels
  tels que
  \begin{conditions}
  \item L'ensemble $\{\,w\in\Val(\KK)\mid \lambda_w\neq 1\,\}$
    est fini;
  \item On a la relation $\Vert\varphi(y)\Vert'_w\leq\lambda_w\Vert y\Vert_w$
    pour tout $w\in\Val(\KK)$, tout $x\in V(\KK_w)$ et tout $y\in E(x)$.
  \end{conditions}
\end{prop}
\begin{proof}
  Soit $\PP(E)$ le fibr\'e projectif sur~$V$ correspondant aux
  droites des fibres de~$E$. Pour tout $v\in T$, on consid\`ere l'application
  $\rho_v:\PP(E)(\KK_v)\to\RRpp$ d\'efinie par
  la relation $\Vert\phi(y)\Vert'_v=\rho_v(x)\Vert y\Vert_v$ pour
  tout $x\in \PP(E)(\KK_v)$ et tout $y$ non nul appartenant \`a la droite
  correspondante de $E$.
  Comme $\PP(E)(\KK_v)$ est compact, l'application~$\rho_v$ est major\'ee et
  atteint sa borne sup\'erieure qu'on note $\lambda_v$. Pour presque
  toute place finie~$v$ de~$\KK$, le morphisme~$\phi$ est induit
  par un morphisme de $\mathcal E$ vers $\mathcal F$, pour un
  mod\`ele $\mathcal E$ (\resp~$\mathcal F$) de~$E$ (\resp~$F$)
  d\'efinissant la norme consid\'er\'ee. Pour toute telle place,
  on a $\lambda_v\leq 1$. Cela prouve la proposition.
\end{proof}
\begin{nota}
  \label{nota.norme}
  Sous les hypoth\`eses de la proposition, pour toute place~$v$ de~$\KK$
  et tout~$x$ de $V(K_v)$, on note $\vvert\varphi_x\vvert_v$
  le plus petit nombre r\'eel $\lambda\geq 0$ tel que
  \[\Vert \varphi(y)\Vert_v\leq\lambda \Vert y\Vert_v\]
  pour $y\in E(x)$. Si~$x$ appartient \`a $V(\KK)$, on pose
  \'egalement
  \[\vvert\varphi_x\vvert=
  \prod_{v\in\Val(\KK)}\vvert\varphi_x\vvert_v.\]
\end{nota}
\subsection{Hauteurs d'Arakelov}
\label{subsection.hauteur}
Rappelons maintenant comment un fibr\'e en droites
muni d'une norme ad\'elique permet
de d\'efinir une hauteur sur une vari\'et\'e.

\begin{defi}
Soit~$\KK$ un corps de nombres et~$V$ une belle vari\'et\'e sur~$\KK$.
On appelle \emph{hauteur d'Arakelov}%
\index{Hauteur>darakelov=d'Arakelov}
sur~$V$ un fibr\'e en droites sur~$V$ muni d'une norme ad\'elique.
Soit~$L$ une hauteur d'arakelov sur~$V$ et $x\in V(\KK)$ un point rationnel
de~$V$. L'image inverse $x^*(L)$ de~$L$ par~$x$ (exemple~\ref{exem.4.4.h})
d\'efinit un \'el\'ement du groupe $\mathcal H(\Spec(\KK))$. La 
\emph{hauteur logarithmique de~$x$ relativement \`a~$L$}%
\index{Hauteur>danslecasarithm\'etique=dans le cas arithm\'etique}
est d\'efinie par
\[h_L(x)=\dega(x^*(L)).
\glossary{$h_L$}\]
La \emph{hauteur exponentielle de~$x$ relativement \`a~$L$}%
\index{Hauteur>exponentielle}
est d\'efinie par
\[H_L(x)=\exp(h_L(x)).
\glossary{$H_L$}\]
\end{defi}
\begin{exem}
  \label{exem.proj.naif}
  Munissons $\RR^{n+1}$ de sa structure euclidienne usuelle
  et, pour un nombre premier $p$, l'espace vectoriel
  $\QQ_p^{n+1}$ de la norme usuelle:
  \[\Vert (x_0,\dots,x_n)\Vert_p=\max_{0\leq i\leq n}(|x_i|_p).\]
  En identifiant l'ensemble des points de l'espace projectif
  sur un corps $\KK$ \`a celui des droites de $\KK^{n+1}$,
  l'espace tangent en la droite $L$ peut s'identifier
  au quotient $L\dual\otimes\KK^{n+1}/L\dual\otimes L$.
  Les normes pr\'ec\'edentes induisent par restriction,
  produit tensoriel et quotient une m\'etrique ad\'elique
  sur l'espace projectif, et donc une norme ad\'elique
  sur $\omega_{\PP^n_\QQ}^{-1}$. La hauteur correspondante
  est donn\'ee par
  \[H_n(x)=\dega(T_x\PP^n(\QQ))={\sqrt{\sum_{i=0}^nx_i^2}\rlap{.}}^{\,n+1}\]
  si $x=(x_0:\dots:x_n)$ avec $x_0,\dots,x_n$ des entiers relatifs
  premiers entre eux dans leur ensemble.
\end{exem}

\section{Pentes \`a la Bost}%
\label{section.pentes}
\subsection{Pentes d'un r\'eseau ad\'elique}
Nous allons tout d'abord rappeler la d\'efinition des pentes
arithm\'etiques pour un r\'eseau ad\'elique.
\begin{defi}
Soit~$E$ un $\KK$-espace vectoriel~$E$
muni d'une part d'un $\mathcal O_\KK$-sous-module $M$ projectif de rang
constant \'egal \`a la dimension de~$E$ et, pour toute place 
archim\'edienne~$w$ de~$\KK$, d'une norme classique $\Vert\cdot\Vert_w$
sur $E\otimes_\KK\KK_w$. Soit $n$ la dimension de~$E$, la construction 
de l'exemple \ref{exem.4.4.g}
munit alors $\det(E)=\exterieur^nE$ d'une norme ad\'elique
et on d\'efinit le \emph{degr\'e arithm\'etique de~$E$} par la relation
\[\dega(E)=\dega(\det(E)),\]
o\`u la d\'efinition dans le cas de la dimension~1 est donn\'ee
par la formule~\eqref{equa.pente.droite}.
Par d\'efinition, le degr\'e arithm\'etique de l'espace vectoriel
nul est~$0$.
\end{defi}
\begin{remas}
Ce degr\'e peut \'egalement \^etre d\'ecrit de la fa\c con suivante:
les normes d\'efinissent une norme euclidienne sur le $\RR$-espace vectoriel
$E_\RR=E\otimes_\QQ\RR$ qui est isomorphe \`a la somme de $\RR$-espaces 
vectoriels $\bigoplus_{w\mid\infty}E\otimes_\KK\KK_w$. Notons $\iota$ le plongement
canonique de $E$ dans $E_\RR$.
Le module~$\iota(M)$ est un r\'eseau de $E_\RR$ et le degr\'e est
donn\'e par
\[\dega(E)=-\log(\Vol(E_\RR/\iota(M)))\]
o\`u $\Vol(E_\RR/\iota(M))$ est, par d\'efinition, le volume 
euclidien d'un domaine fondamental
pour le r\'eseau $\iota(M)$.
\end{remas}
\begin{exem}
Pla\c cons-nous dans le cas o\`u $\KK=\QQ$, soit $E$ un $\QQ$-espace vectoriel
muni d'un r\'eseau~$\Lambda$ et d'une norme euclidienne~$\Vert\cdot\Vert$. 
Pour toute droite 
vectorielle~$F$ de~$E$, le degr\'e arithm\'etique
de~$F$ pour la structure induite
est $-\log(\Vert f\Vert)$ o\`u $f$ est un des deux g\'en\'erateurs
de l'intersection $\Lambda\cap F$.
\end{exem}
\begin{defi}
\'Etant donn\'e un r\'eseau ad\'elique $E$ de dimension~$n$, 
le \emph{polyg\^one de Newton}
$\mathcal P(E)$ de~$E$ est l'enveloppe convexe des points $(\dim(F),\dega(F))$
du plan o\`u~$F$ d\'ecrit l'ensemble des sous-espaces vectoriels de~$E$
munis des structures ad\'eliques induites. Le covolume des sous-r\'eseaux
de~$E$ \'etant minor\'e, le polyg\^one est born\'e sup\'erieurement
par le graphe d'une application affine par morceaux et on d\'efinit
une application $m_E$ de $[0,n]$ dans $\RR$ par la relation
\begin{equation}
\label{equ.2}
m_E(x)=\sup\{y\in\RR\mid (x,y)\in\mathcal P(E)\}
\end{equation}
pour tout~$x$ de $[0,n]$. Les \emph{pentes arithm\'etiques}\index{Pentes}
de~$E$ sont alors donn\'ee par
\[\mu_i(E)=m_E(i)-m_E(i-1)\]
pour $i\in\{1,\dots,n\}$. Ce sont les pentes des segments
formant le graphe de $m_E$.
\end{defi}
\begin{listrems}
  \remarque
  Notons que nous avons les in\'egalit\'es
  \[\mu_n(E)\leq \mu_{n-1}(E)\leq\dots\leq \mu_1(E)\]
  et la relation
  \[\sum_{i=1}^n\mu_i(E)=\dega(E).\]
  \remarque
  \label{rema.pente.dualite}
  On a \'egalement la relation $\mu_i(E)=-\mu_{n+1-i}(E\dual)$
  pour $1\leq i\leq\dim(E)$ o\`u~$E\dual$ d\'esigne l'espace
  dual de~$E$ \cite[(3.5)]{bostkuennemann:hermitianI}.
  \remarque
  Dans le cadre du programme de Batyrev et Manin, il est naturel
  d'utiliser des hauteurs qui ne sont pas normalis\'ees, et donc
  ne sont pas invariantes par extension de corps.
  De fa\c con coh\'erente, nous n'utilisons pas les pentes normalis\'ees: contrairement
  \`a Bost \cite[p. 195]{bost:leaves} nous ne divisons pas
  par le degr\'e du corps de nombres. Ces pentes
  ne sont donc pas invariantes par extension de corps.
\end{listrems}
\begin{nota}
  Avec les notations de la d\'efinition pr\'ec\'edente,
  on notera \'egalement $\mu_{\max}(E)$ pour $\mu_1(E)$ et
  $\mu_{\min}(E)$ pour $\mu_n(E)$. La \emph{pente de $E$} est la
  moyenne des pentes: $\mu(E)=\frac{\dega(E)}{\dim(E)}$.
\end{nota}
\begin{listrems}
  \remarque
  La notion de pentes a \'et\'e g\'en\'eralis\'ee par E.~Gaudron
  aux normes non classiques \cite{gaudron:pentes}.
  
  \remarque
  \label{rema.espaces.minima}
  Le $i$-\`eme minima de~$E$, not\'e $\lambda_i(E)$
  peut \^etre d\'efini comme la borne inf\'erieure
  des nombres $\theta\in\RRpp$ tels qu'il existe une famille libre
  $(x_1,\dots,x_i)$ de~$E$ v\'erifiant les in\'egalit\'es
  \[\prod_{w\in\Val(\KK)}\Vert x_i\Vert_w\leq \theta\]
  pour $i\in\{1,\dots,n\}$ (\cf\ \cite{thunder:adelic}).
  Le  th\'eor\`eme de Minkowski permet d'obtenir
  une constante explicite $C_\KK$
  de sorte que
  \[0\leq \log(\lambda_i(E))+\mu_i(E)\leq C_\KK\]
  pour $i\in\{1,\dots, n\}$; cet encadrement
  est donn\'e par E.~Gaudron dans \cite[th\'eor\`eme 5.20]{gaudron:pentes},
  compte tenu du fait que $\Delta(\overline E)=1$ dans
  notre cas, avec une r\'ef\'erence \`a un travail
  de T.~Borek \cite{borek:minima}.
\end{listrems}
\subsection{Pentes d'un point rationnel}
Nous allons maintenant appliquer cette construction aux points rationnels
d'une vari\'et\'e
\begin{defi}
  Soit~$V$ une belle vari\'et\'e de dimension~$n$
  sur le corps de nombres~$\KK$.
  Soit~$E$ un fibr\'e ad\'elique sur la vari\'et\'e~$V$. Soit~$r$
  le rang du fibr\'e~$E$. \'Etant donn\'e
  un point rationnel~$x$ de~$V$, les \emph{pentes de~$x$ relativement \`a~$E$}
  sont donn\'ees par la formule $\mu_i^E(x)=\mu_i(E(x))$ o\`u $i\in\{1,\dots,r\}$.
\par
En particulier, lorsque la vari\'et\'e~$V$ est muni d'une m\'etrique
ad\'elique, les \emph{pentes de $x$} sont les pentes arithm\'etiques 
de l'espace tangent $T_xV$. On les note simplement $\mu_i(x)$
pour $1\leq i\leq n$.
\end{defi}
\begin{listrems}
  \remarque
  La somme $\sum_{i=1}^n\mu_i(x)$ est le degr\'e arithm\'etique de l'espace
  tangent $T_xV$ c'est \`a dire le degr\'e arithm\'etique de la fibre en~$x$ du
  fibr\'e anticanonique, qui n'est rien d'autre que la hauteur logarithmique
  de~$x$ relativement \`a ce fibr\'e.
  
  \remarque
  La d\'efinition des pentes dans le cadre arithm\'etique
  est un analogue de celle introduite pour les fibr\'es vectoriels
  sur les courbes alg\'ebriques.
  Plus pr\'ecis\'ement, soit~$\mathcal C$ une courbe
  projective, lisse et g\'eom\'etriquement
  int\`egre sur un corps~$k$ et soit~$V$ une bonne vari\'et\'e sur~$k$.
  \'Etant donn\'e un morphisme
  de vari\'et\'es~$\varphi$ de~$\mathcal C$ dans~$V$,
  le polyg\^one d'Harder-Narasimhan du fibr\'e vectoriel $\phi^*(TV)$
  est d\'efini comme l'enveloppe convexe des couples $(\rg(F),\deg(F))$
  o\`u~$F$ d\'ecrit l'ensemble des sous-fibr\'es de $\phi^*(TV)$.
  Les pentes d'un point rationnel $\mu_i(x)$ d\'efinies ici correspondent
  donc aux pentes $\mu_i(\phi^*(TV))$. Dans le cas particulier o\`u
  la courbe $\mathcal C$ est la droite projective sur~$k$,
  le fibr\'e $\phi^*(TV)$ est isomorphe \`a une
  unique somme directe
  $\bigoplus_{i=1}^n\mathcal O(a_i)$
  avec $a_1\geq a_2\geq\dots \geq a_n$ et on a les \'egalit\'es
  $\mu_i(\phi^*(TV))=a_i$ pour $1\leq i\leq n$.
\end{listrems}
\begin{nota}
  Avec les notations de la d\'efinition pr\'ec\'edente, on d\'efinit
  \[h(x)=\sum_{i=1}^n\mu_i(x)\qquad\text{et}\qquad H(x)=\exp(h(x))\]
  pour tout point rationel $x$ de $V$.
  La fonction $H:V(\KK)\to\RR$ ainsi d\'efinie est donc une hauteur
exponentielle relative au fibr\'e anticanonique.
\par
On pose \'egalement $\mu_{\min}(x)=\mu_n(x)$,
$\mu_{\max}(x)=\mu_1(x)$ et $\mu(x)=\mu(T_xV)=h(x)/n$.
\par
\end{nota}
\subsection{Libert\'e d'un point rationnel}
Nous allons utiliser le pentes d\'efinies au paragraphe pr\'ec\'edent pour
d\'efinir une notion de libert\'e pour les points rationnels. Cette notion
est inspir\'ee de la notion de courbe tr\`es libre. De mani\`ere plus
pr\'ecise, le fait d'avoir une libert\'e strictement positive pour
un point rationnel correspond au fait d'\^etre tr\`es libre pour un 
morphisme de la droite projective dans la vari\'et\'e. Notons que cette
notion de mesure de la libert\'e, qui, \`a la connaissance de l'auteur,
n'a pas \'et\'e formalis\'e dans le cadre g\'eom\'etrique, 
pourrait \'egalement avoir un int\'er\^et dans ce cadre-l\`a. 
\begin{defi}
Soit~$V$ une belle vari\'et\'e de dimension~$n>0$ sur le corps de nombres~$\KK$.
On suppose que la vari\'et\'e~$V$ est munie d'une m\'etrique ad\'elique.
Pour tout point rationnel~$x$ de~$V$, la libert\'e de~$x$ est le nombre
r\'eel
\[l(x)=\frac{\mu_{\min}(x)}{\mu(x)}\]
si $\mu_{\min}(x)>0$ et vaut~$0$ sinon.
\end{defi}
\begin{listrems}
\remarque Soit~$x$ un point rationnel de~$V$. Rappelons
que $\mu(x)=\frac{h(x)}n$ n'est rien d'autre que la moyenne des pentes.

\remarque Par d\'efinition, $l(x)\in[0,1]$.

\remarque La libert\'e $l(x)$ est nulle si et seulement si la
pente minimale est n\'egative.

\remarque On a l'\'egalit\'e $l(x)=1$ si et seulement toutes
les pentes sont \'egales:
\[\mu_1(x)=\mu_2(x)=\dots=\mu_n(x)=\frac{h(x)}n.\]
Cela revient \`a dire que l'espace tangent $T_xV$ est semi-stable
(\cf \cite[\S 1.2]{bostchen:semistability}). C'est par exemple
le cas sur $\QQ$ si le r\'eseau dans $T_xV$ est engendr\'e par une
base orthonormale.

\remarque La libert\'e d'un point rationnel d\'epend du choix
de la m\'etrique sur~$V$. Nous verrons dans le paragraphe suivant
que cette d\'ependance diminue avec la hauteur du point consid\'er\'e.

\remarque Bien que nos conventions font que les pentes d'un point
ne sont pas stables par extension de corps, la libert\'e d'un point rationnel
est, quant \`a elle, stable par extension de corps.

\remarque
La remarque \ref{rema.espaces.minima} fournit une constante~$C$
de sorte que
\[|\mu_n(x)-\log(\lambda_1(T_xV\dual))|\leq C.\]
Par cons\'equent, \`a un terme born\'e pr\`es,
la condition  $l(x)\leq \epsilon$
pour un point $x\in V(\KK)$ peut \^etre vue comme l'existence
d'un \'el\'ement $y\in T_xV\dual$ tel que
\[\prod_{w\in\Val(\KK)}\Vert y\Vert_w\leq H(x)\rlap{.}^\epsilon \]
\end{listrems}
L'objectif du reste de cet article est de motiver le slogan suivant
\begin{slo}
Les mauvais points d'une vari\'et\'e du point de vue
de l'\'equidistribution
sont ceux dont la libert\'e est r\'eduite.
\end{slo}
\section{Propri\'et\'es \'el\'ementaires}%
\label{section.elementaires}
\subsection{D\'ependance de la libert\'e en la m\'etrique}
Nous commen\c cons par un r\'esultat \'el\'ementaire
concernant les morphismes de fibr\'es.
\begin{lemm}
  \label{lemm.maj.morph}
Soit~$V$ une belle vari\'et\'e sur le corps de nombres~$\KK$,
Soient~$E$ et~$F$ des fibr\'es vectoriels ad\'eliquement norm\'es
au-dessus de~$V$ et soit~$\varphi$ un morphisme de~$E$ dans~$F$.
Alors, il existe une constante r\'eelle $C\geq 0$
telle que, pour tout point~$x$ de $V(\KK)$ en lequel
l'application induite $\varphi_x:E(x)\to F(x)$ est surjective,
on a
\[m_{F(x)\dual}(i)\leq m_{E(x)\dual}(i)+C\]
pour $i\in\{0,\dots,\rg(F)\}$. En particulier,
\[\mu_{\min}(E(x))\leq\mu_{\min}(F(x))+C.\]
\end{lemm}
\begin{proof}
  Par dualit\'e, le morphisme $\varphi$ d\'efinit un morphisme
  $\varphi\dual$ de $F\dual$ dans $E\dual$.
  Soit~$x$ un \'el\'ement de $V(\KK)$
  en lequel $\varphi_x$ est surjectif. L'application
  lin\'eaire $\varphi_x\dual$ est injective.
  Pour tout sous-espace~$H$ de $F(x)\dual$
  de dimension~$k$ et tout $y\in\exterieur^kH$,
  en utilisant les notations~\ref{nota.norme}, on a les relations
  \[\Vert\wedge^k\varphi(y)\Vert_w\leq\vvert \wedge^k\varphi\dual_x\vvert_w
  \Vert y\Vert_w\]
  pour $w\in\Val(\KK)$, ce qui implique l'in\'egalit\'e
  \[\dega(\varphi_x\dual(H))\geq \dega(H)
  -\log(\vvert \wedge^k\varphi\dual_x\vvert).\]
  En posant $C_x=\max_{1\leq k\leq \rg(F)}(\log(\vvert \wedge^k\varphi\dual_x
  \vvert))$, on en d\'eduit que
  \[m_{F(x)\dual}(i)\leq m_{E(x)\dual}(i)+C_x\]
  pour $1\leq i\leq\rg(F)$. Compte tenu de la
  proposition~\ref{prop.maj.morph}, il
  existe une constante $C\geq0$ telle que $C_x\leq C$
  pour tout $x\in V(K)$. En particulier, on obtient l'in\'egalit\'e
  \[\mu_1(F(x))\dual)\leq \mu_1(E(x)\dual)+C.\]
  La formule de la remarque~\ref{rema.pente.dualite} donne
  alors l'in\'egalit\'e
  \[\mu_{\min}(E(x))\leq\mu_{\min}(F(x))+C\]
  ce qui permet de conclure.
\end{proof}
\begin{rema}
  Notons que la preuve donne en fait une in\'egalit\'e plus
  pr\'ecise:
  \[\mu_{\min}(E(x))\leq\mu_{\min}(F(x))
  +\max_{1\leq k\leq \rg(F)}\Bigl(
  \frac{\log(\vvert \wedge^k\varphi\dual_x
  \vvert)}k\Bigr).\]
\end{rema}
Comme annonc\'e, nous allons maintenant d\'emontrer que la libert\'e
d\'epend peu de la m\'etrique choisie.
\begin{lemm}
\label{lemm.5.2}
Soit~$V$ une belle vari\'et\'e sur le corps de nombres~$\KK$,
soit~$E$ un fibr\'e vectoriel de rang~$r$ sur~$V$. Donnons-nous
des normes ad\'eliques classiques
${(\Vert\cdot\Vert_w)_{w\in\Val(\QQ)}}$ et ${(\Vert\cdot\Vert'_w)_{w\in\Val(\QQ)}}$
sur~$E$.
On note $\mu_i^E$ et $\mu_i^{E'}$ les fonctions de pentes associ\'ees.
Il existe une constante r\'eelle $C\geq 0$ telle que
\[|\mu_i^E(x)-\mu_i^{E'}(x)|\leq C\]
pour $i\in\{1,\dots,r\}$ et $x\in V(\KK)$.
\end{lemm}
\begin{proof}
  En appliquant le lemme~\ref{lemm.maj.morph} \`a l'endomorphisme
  identit\'e, on obtient que les
  applications $x\mapsto m'_{E(x)}(i)-m_{E(x)}(i)$
  sont born\'ees. Le lemme en r\'esulte.
\end{proof}
\begin{prop}
\label{prop.5.2}
Soit~$V$ une belle vari\'et\'e sur le corps~$\KK$,
munie de deux m\'etriques. On note~$l$ et~$l'$ les fonctions
libert\'es associ\'ees \`a ces m\'etriques et $h$ la fonction hauteur
d\'efinie par la premi\`ere d'entre elles.
Alors il existe une constante r\'eelle $C>0$ telle que
\[|l(x)-l'(x)|<\frac C{h(x)}\]
pour tout $x\in V(\KK)$ tel que $h(x)>0$.
\end{prop}
\begin{proof}
Notons $\mu_i$ (resp. $\mu'_i$) les fonctions de pente
correspondant \`a la premi\`ere m\'etrique (resp. la seconde).
On primera de m\^eme la notation pour la hauteur.
Il r\'esulte du lemme~\ref{lemm.5.2} qu'il existe des constantes
r\'eelles strictement positives $C_1$ et $C_2$ telles que
\[|\mu_n(x)-\mu'_n(x)|\leq C_1\qquad\text{et}\qquad|h(x)-h'(x)|\leq C_2\]
pour $x\in V(\KK)$. Par d\'efinition, on a la majoration
$|l(x)-l'(x)|\leq 1$, il suffit donc de d\'emontrer le r\'esultat lorsque
$h(x)>2C_2$. Notons qu'on a alors les in\'egalit\'es
$h'(x)>h(x)-C_2>h(x)/2$ et,
par cons\'equent, ${1/h'(x)<2/h(x)}$. 
Si $\mu'_n(x)\leq 0$ et $\mu_n(x)\leq 0$, on a $l(x)=l'(x)=0$ ce qui
d\'emontre le r\'esultat. Si $\mu'_n(x)<0$ et $\mu_n(x)>0$,
alors on obtient les relations $l(x)=n\mu_n(x)/h(x)\leq nC_1/h(x)$.
Si $\mu'_n(x)>0$ et $\mu_n(x)<0$, on obtient
de m\^eme $l'(x)\leq 2nC_1/h(x)$. Enfin
si les deux pentes minimales sont strictement positives, 
on a les in\'egalit\'es
\[|l(x)-l'(x)|=n\left|\frac{\mu_n(x)}{h(x)}-\frac{\mu'_n(x)}{h'(x)}\right|
\leq\frac{nC_1}{h(x)}+\frac{n\mu'_n(x)C_2}{h(x)h'(x)}\leq
\frac{nC_1+C_2}{h(x)}.\quad\qed\]
\noqed
\end{proof}
\Subsection{Libert\'e et morphismes de vari\'et\'es}
\begin{prop}
  \label{prop.inegalite.morph.var}
  Soient~$X$ et~$Y$ de belles vari\'et\'es sur~$\KK$
  de dimensions respectives~$m$ et~$n$ et soit
  $\varphi:X\to Y$ un morphisme de vari\'et\'es.
  Alors il existe une constante $C\geq 0$ telle qu'en tout
  $x\in X(\KK)$ en lequel l'application tangente $T_x\varphi$ est
  surjective,
  \[\mu_{\min}(x)\leq\mu_{\min}(\varphi(x))+C.\]
  Si, en outre, $h(x)$ est strictement positif,
  Il en r\'esulte que
  \[l(x)\leq \frac{m\,h(\varphi(x))}{n\,h(x)}l(\varphi(x))+
  \frac{mC}{h(x)}.\]
\end{prop}
\begin{proof}
  La premi\`ere assertion r\'esulte du lemme~\ref{lemm.maj.morph} appliqu\'e
  au morphisme tangent ${T\varphi:TX\to\varphi^*(TY)}$,
  la seconde de la d\'efinition de la libert\'e.
\end{proof}
\begin{rema}
  \label{rema.morph.liberte}
  Il r\'esulte de cette proposition que, si $y\in Y(\KK)$ et $X_y$
  d\'esigne la fibre de~$\varphi$ au-dessus de~$y$,
  alors la liberte $l(x)$ converge vers $0$ lorsque $h(x)$ tend vers
  $+\infty$ dans $X_y(\KK)$ en dehors des points critiques.
\end{rema}
Le r\'esultat suivant montre le lien
entre la libert\'e d'une courbe rationnelle et celle de ses
points.
\begin{nota}
  Soit $\varphi:\PP^1_\KK\to V$ un morphisme de vari\'et\'es,
  le fibr\'e $\varphi^*(T_V)$ est isomorphe \`a un
  unique fibr\'e de la forme $\bigoplus_{i=1}^n\mathcal O_{\PP^1_\KK}(a_i)$
  avec $a_1\geq a_2\geq\dots\geq a_n$. On note $\mu_i(\varphi)=a_i$
  pour $i\in{1,\dots,n}$ et $\deg(\varphi)=\sum_{i=1}^n\mu_i(\varphi)$.
  On d\'efinit alors la libert\'e de $\varphi$ comme le
  nombre rationnel $l(\varphi)=n\mu_n(\varphi)/\deg(\varphi)$
  si $\mu_n(\varphi)>0$.
  On pose $l(\varphi)=0$, si $\mu_n(\varphi)<0$.
\end{nota}
\begin{prop}
  \label{prop.courbes}
  Soit $\varphi:\PP^1_\KK\to V$ un morphisme de vari\'et\'es non constant,
  alors
  \begin{assertions}
  \item
    Si $\mu_n(\varphi)<0$, alors $l(\varphi(x))=0$ pour tout $x$
    de $\PP^1(\KK)$ en dehors d'une partie finie;
  \item
    Si $\mu_n(\varphi)\geq0$, alors $l(\varphi(x))$ converge vers~$l(\varphi)$
    dans $\PP^1(\KK)$ pour le filtre de Fr\'echet du compl\'ementaire
    des parties finies.
  \end{assertions}
\end{prop}
\begin{proof}
  On peut munir $\varphi^*(TV)$ de deux m\'etriques: d'une part
  de celle d\'eduite par image inverse de la m\'etrique de~$V$,
  et d'autre part de celle issue de l'isomorphisme de $\varphi^*(TV)$
  sur $\bigoplus_{i=1}^n\mathcal O_{\PP^1_\KK}(\mu_i(\varphi))$.
  D'apr\`es le lemme~\ref{lemm.5.2}, on en d\'eduit
  que
  \[\mu_n(\varphi(x))=\mu_n^{\varphi^*(TV)}(x)=\mu_n(\varphi)h_1(x)+O(1)\]
  et
  \[h(\varphi(x))=\deg(\varphi)h_1(x)+O(1)\]
  o\`u $h_1(x)$ d\'esigne une hauteur sur $\PP^1_\KK$ relative au
  fibr\'e $\mathcal O_{\PP^1_\KK}(1)$.
  L'assertion a) en r\'esulte par d\'efinition de la libert\'e.
  D'autre part, comme le morphisme $\varphi$ est non constant,
  l'application d\'eriv\'ee de $\varphi$ donne un morphisme non
  trivial de $\mathcal O_{\PP^1_\KK}(2)$ dans $\varphi^*(TV)$
  ce qui prouve que $\mu_1(\varphi)\geq 2$. En particulier, dans le cas b),
  ${\deg(\varphi)>0}$ ce qui compl\`ete la
  preuve.
\end{proof}
\section{Empirisme}%
\label{section.formule}
\subsection{Rappels sur la constante empirique}
Nous allons rappeler ici l'interpr\'etation envisag\'ee
pour la constante qui appara\^\i t dans le terme dominant
du nombre de bons points rationnels de hauteur born\'ee.
Une premi\`ere version de cette constante a d'abord \'et\'e
d\'efinie dans~\cite{peyre:fano}. Par la suite, Batyrev
et Tschinkel dans \cite{batyrevtschinkel:toric} l'ont corrig\'ee en rajoutant
le facteur entier $\beta(V)$. Depuis, une r\'einterpr\'etation
de cette constante par Salberger dans \cite{salberger:tamagawa}
et les nombreux exemples connus ont confirm\'e son int\'er\^et.
D'autre part, Batyrev et Tschinkel l'ont g\'en\'eralis\'ee dans
le cas o\`u la hauteur n'est pas relative au fibr\'e anticanonique
\cite{batyrevtschinkel:tamagawa}
mais cette g\'en\'eralisation sort du cadre de cet article.

Dans ce paragraphe, la lettre~$V$ d\'esigne une belle
vari\'et\'e munie d'une m\'etrique ad\'elique.
Pour pouvoir d\'efinir la constante, nous
supposerons en outre que
\begin{conditions}
  \item Les groupes de cohomologie $H^i(V,\mathcal O_V)$ sont
    nuls pour $i=1$ ou $2$;
  \item Le groupe de Picard g\'eom\'etrique de~$\overline V$
    est un $\ZZ$-module libre de rang fini.
\end{conditions}
Pour toute place~$v$ de~$\KK$, on note $\haar {x_v}$ la mesure de Haar
sur~$\KK_v$ normalis\'ee de la fa\c con suivante:
\begin{itemize}
\item Si $v$ est une place r\'eelle, alors $\haar {x_v}$ est la
  mesure de Lebesgue usuelle sur $\RR$;
\item Si $v$ est une place complexe, alors
  $\haar {x_v}=\ci\haar{z}\haar{\overline z}=2\haar x\haar y$;
\item Sinon $\int_{\mathcal O_v}\haar{x_v}=1$.
\end{itemize}
La m\'etrique ad\'elique sur~$V$ induit une norme ad\'elique
$(\Vert\cdot\Vert_v)_{v\in\Val(\KK)}$ sur le faisceau anticanonique.
Rappelons la construction de la mesure de Tamagawa \`a partir
de cette norme ad\'elique.
Pour tout place~$v$ de~$\KK$, on peut alors d\'efinir une mesure
bor\'elienne sur l'espace $V(\KK_v)$ donn\'ee dans
un syst\`eme de coordonn\'ees locales
${(x_1,x_2,\dots,x_n)}$ par la formule
\[\oomega_v=\left\Vert\frac\partial{\partial x_1}\wedge
\dots\wedge\frac\partial{\partial x_n}\right\Vert_v
\haar {x_{1,v}}\dots\haar {x_{n,v}}.\]
La formule de changement de variables \cite{weil:adeles} permet de montrer
que cette expression est bien ind\'ependante du syst\`eme de
coordonn\'ees choisi.
\begin{lemm}
  \label{lemme.mesure.finie}
  Soit~$v$ une place de~$\KK$.
  Supposons que la norme $\Vert\cdot\Vert_v$ est d\'efinie
  par un mod\`ele projectif et lisse $\mathcal V$ de~$V$ sur l'anneau
  des entiers $\mathcal O_v$ de $\KK_v$. Soit $\mathfrak m_v$
  l'id\'eal maximal de $\mathcal  O_v$ et
  $\FF_v$ le corps $\mathcal O_v/\mathfrak m_v$; pour tout entier~$k$
  notons
  \[\pi_k:V(\KK_v)\longrightarrow \mathcal V(\mathcal O_v/\mathfrak m_v^k)\]
  l'application de r\'eduction modulo~$\mathfrak m_v^k$. Alors $\oomega_v$
  est la \emph{mesure naturelle} caract\'eris\'ee par
  les relation
  \begin{equation}\label{equ.mesure.can}
    \oomega_v(\pi_k^{-1}(X))=\frac{\card X}{\card{\FF_v^{nk}}}
  \end{equation}
  o\`u $k$ parcourt les entiers strictement positifs
  et $X$ les parties de l'ensemble fini
  $\mathcal V(\mathcal O_v/\mathfrak m_v^k)$.
  En particulier, on a l'\'egalit\'e
  \[\oomega_v(V(\KK_v))=d_v(V)=\frac{\card \mathcal V(\FF_v)}{\card{\FF_v^n}}.\]
\end{lemm}
\begin{proof}
  Cette propri\'et\'e est locale,
  il suffit donc de d\'emontrer la formule~\eqref{equ.mesure.can}
  dans le cas o\`u $X$ est un singleton. Comme $\mathcal V$ est
  suppos\'ee projective et lisse sur $\mathcal O_v$,
  on choisit un plongement $\mathcal V\to\PP^N_{\mathcal O_v}$
  de sorte que les fonctions rationnelles
  $x_1=\frac{X_1}{X_0},\dots,x_n=\frac{X_n}{X_0}$
  forment un syst\`eme de coordonn\'ees
  sur l'image inverse de $X$ dans $\mathcal V(\mathcal O_v)$ et que
  $(\frac{\partial}{\partial x_1},\dots,\frac{\partial}{\partial x_n})$
  donne la $\mathcal O_v$-structure sur le
  fibr\'e tangent $T\mathcal V_{\mid W}$.
  L'application $f:W\to \KK_v^n$ d\'efinie par $(x_1,\dots,x_n)$
  pr\'eserve alors la congruence modulo $\mathfrak m_v^k$ et
  $\frac{\partial}{\partial x_1}\wedge\dots\wedge\frac{\partial}{\partial x_n}$
  est de norme~$1$ sur~$W$. Par cons\'equent, il existe $x_0\in \KK_v^n$
  de sorte
  que
  \[\oomega_v(\pi_k^{-1}(X))=\int_{x_0+\mathfrak m_v^k}\haar {x_{1,v}}\dots
  \haar{x_{n,v}}=\card\FF_v^{-nk}.\qed\]\noqed
\end{proof}
Comme dans \cite[\S2.1]{peyre:fano}, il r\'esulte alors
de la formule de Lefschetz
et de la conjecture de Weil d\'emontr\'ee par Deligne~\cite{deligne:weil}, que,
sous les hypoth\`eses (i) et~(ii) faites ci-dessus,
on a pour presque toute place~$v$ de~$\KK$ la formule
\[\oomega_v(V(\KK_v))=d_v(V)=1+
\frac{1}{\card\FF_v}\Tr(\Fr_v|\Pic(\mathcal V_{\overline{\FF}_v}\otimes\QQ))+ O(\card\FF_v^{-3/2}).\]
Suivant le principe de construction des mesures de Tamagawa
sur les espaces ad\'eliques, cela am\`ene \`a la renormalisation suivante:
\begin{defi}
  On fixe une extension galoisienne~$\LL$ de~$\KK$ qui d\'eploie le grou\-pe
  de Picard de~$V$.
  Notons~$S$ une partie finie de $\Val(\KK)$ contenant l'ensemble
  des places archim\'ediennes ainsi que les places se ramifiant dans
  l'extension $\LL/\KK$. Pour tout $v\in\Val(\KK)\setminus S$,
  La substitution de Frobenius correspondant \`a une extension~$w$
  de~$v$ \`a~$\KK$ sera not\'ee $(w,\LL/\KK)$ (\cf\cite[\S1.8]{serre:corpslocaux}).
  On pose alors, pour tout nombre complexe $s$ dont la partie r\'eelle
  $\Re(s)$ est strictement positive,
  \[L_v(s,\Pic(\overline V))=\frac{1}{\Det(1-\card\FF_v^{-s}(w,\LL/\KK)\mid
    \Pic(\overline V))}.\]
  et, pour tout nombre complexe $s$ tel que $\Re(s)>1$,
  \[L_S(s,\Pic(\overline V))=\prod_{v\in\Val(\KK)-S}L_v(s,\Pic(\overline V)),\]
  se produit convergeant d'apr\`es \cite[th\'eor\`eme 7]{Artin:lreihen}.
  La mesure de Tamagawa sur l'espace ad\'elique $V(\Adeles_\KK)$
  est alors d\'efinie par
  \[\oomega_V=\frac{\lim_{s\to 1}(s-1)^t
    L_S(s,\Pic(\overline V)}{\sqrt{d_\KK}^{\,\dim(V)}}
  \prod_{v\in\Val(\KK)}\frac 1{L_v(1,\Pic(\overline V))}\oomega_v.\]
\end{defi}
\begin{rema}
  Notons que cette mesure est ind\'ependante du choix de l'ensemble~$S$.
  Par contre, elle d\'epend de la m\'etrique ad\'elique dont est
  \'equip\'ee~$V$.
\end{rema}
Comme l'a fait remarquer Swinnerton-Dyer, le bon domaine d'int\'egration
pour la valeur de la constante est forc\'ement l'adh\'erence des
points rationnels dans l'espace ad\'elique. Toutefois cette adh\'erence
n'est a priori pas connue, ce qui pourrair emp\^echer de calculer
explicitement la constante dans des cas particuliers. En cons\'equence,
comme c'est maintenant usuel dans l'exploration du programme de Batyrev
et Tschinkel, nous allons supposer implicitement que l'obstruction
de Brauer-Manin~\cite{manin:brauer} \`a l'approximation faible est la seule.
Rappelons rapidement la construction de cette obstruction~(\cf
\cite{peyre:obstruction} pour un survol plus d\'etaill\'e).
On note $\Br(F)=H^2(F,\mathbf G_m)$ le groupe de Brauer d'un corps~$F$.
La th\'eorie du corps de classe global
(\cf \cite[theorem 8.1.17]{neukirchschmidtwingberg}) donne une suite exacte
\begin{equation}
  \label{equ.classe}
  0\longrightarrow\Br(\KK)\longrightarrow\bigoplus_{v\in\Val(\KK)}\Br(\KK_v)@>
  \sum_{v\in\Val(\KK)}\inv_v>>\QQ/\ZZ\longrightarrow 0.
\end{equation}
Notons $\Br(V)=H^2_{\et}(V,\mathbf G_m)$ le groupe de Brauer cohomologique
de~$V$. Pour toute extension~$L$ de~$\KK$ et tout $L$-point~$x$ de~$V$,
la fonctorialit\'e du groupe de Brauer donne un morphisme de sp\'ecialisation
$\ev_x:\Br(V)\to\Br(L)$. On peut donc d\'efinir un accouplement
de $\Br(V)\times V(\Adeles_\KK)$ dans $\QQ/\ZZ$ par
\[\langle A,(x_v)_{v\in\Val(\KK)}\rangle=\sum_{v\in\Val(\KK)}\inv_v(\ev_{x_v}(A))\]
pour $A\in\Br(V)$ et $(x_v)_{v\in\Val(\KK)}\in V(\Adeles_\KK)$.
Autrement dit, pour tout point~$x$ de l'espace des ad\`eles,
on obtient un morphisme de groupes $\omega_x:\Br(V)\to\QQ/\ZZ$ donn\'e
par $A\mapsto \langle A,x\rangle$. 
Compte tenu de la suite exacte~\eqref{equ.classe}, ce morphisme
est trivial si le point~$x$ provient d'un point rationnel et on
appelle $\omega_x$
\emph{l'obstruction de Brauer-Manin en~$x$}. Par continuit\'e
des applications consid\'er\'ees, l'adh\'erence des points rationnels
dans l'espace ad\'elique est contenue dans
\emph{l'espace de Brauer Manin}:
\[V(\Adeles_\KK)^{\Br}=\{\,x\in V(\Adeles_\KK)\mid \omega_x=0\,\}.\]
\begin{defi}
  Rappelons que~$V$ d\'esigne une belle vari\'et\'e munie d'une m\'etrique
  ad\'elique et que~$V$ v\'erifie les conditions (i) et (ii).
  On d\'efinit alors \emph{le nombre de Tamagawa-Brauer-Manin} de~$V$
  par
  \[\tau^{\Br}(V)=\oomega_V(V(\Adeles_\KK)^{\Br}).\]
\end{defi}
Pour compl\'eter la d\'efinition de la constante empirique,
il reste \`a multiplier ce nombre par deux invariants
de la vari\'et\'e, dont nous rappelons maintenant la d\'efinition.
\begin{defi}
  On d\'efinit la constante $\alpha(V)$ par la formule
  \[\alpha(V)=\frac{1}{(t-1)!}\int_{C^1_{\eff}(V)\dual}e^{-\langle \omega_V^{-1},y
    \rangle}\haar y\]
  o\`u on note $C^1_{\eff}(V)$ le c\^one ferm\'e de $\Pic(V)\otimes_\ZZ\RR$
  engendr\'e par les classes de diviseurs effectifs et $C^1_{\eff}(V)\dual$
  le c\^one dual:
  \[C^1_{\eff}(V)\dual=\{\,y\in\Pic(V)\otimes_\ZZ\RR\dual\mid
  \forall x\in C^1_{\eff}(V),\ \langle x,y\rangle\geq 0\,\}.\]
  La mesure sur $\Pic(V)\otimes_\ZZ\RR\dual$ est normalis\'ee
  de fa\c con \`a ce que le r\'eseau d\'efini par $\Pic(V)\dual$
  soit de covolume~$1$.
\end{defi}
\begin{rema}
  Un simple changement de variables montre que cette cons\-tan\-te $\alpha(V)$
  peut \^etre \'egalement vue comme le volume, convenablement normalis\'e,
  de l'intersection du c\^one $C^1_{\eff}(V)\dual$ avec l'hyperplan affine
  d'\'equation $\langle\omega_V^{-1},y\rangle=1$
  (\cf\ \cite[d\'efinition 2.4]{peyre:fano}). On en d\'eduit
  que, dans le cas o\`u le c\^one effectif $C^1_{\eff}(V)$ est
  engendr\'e par un nombre fini de classes de diviseurs effectifs,
  la constante
  $\alpha(V)$ est rationnelle. C'est le cas dans les exem\-ples consid\'er\'es
  ult\'erieurement.
\end{rema}
\begin{defi}
  La constante $\beta(V)$ est l'entier
  \[\beta(V)=\card H^1(\KK,\Pic(\overline V)).\]
\end{defi}
\begin{rema}
  Rappelons que l'introduction de ce terme
  est d\^ue \`a Batyrev et Tschinkel
  \cite{batyrevtschinkel:toric}.
\end{rema}
\begin{defi}
  Dans cet article, la \emph{constante empirique}
  associ\'ee \`a la vari\'et\'e~$V$ munie de sa m\'etrique ad\'elique est
  \[C(V)=\alpha(V)\beta(V)\tau^{\Br}(V).\]
\end{defi}
\begin{rema}
  Salberger dans \cite{salberger:tamagawa}
  a donn\'e une interpr\'etation naturelle
  de la constante en termes des torseurs versels au-dessus
  de la vari\'et\'e, qu'on peut d\'ecrire de la fa\c con suivante:
  les torseurs versels sont munis d'une forme de jauge qui d\'efinit
  une mesure canonique sur l'espace des ad\`eles associ\'e.
  La constante s'interpr\`ete alors en termes de la somme, sur les
  diff\'erentes classes de
  torseurs versels ayant un point ad\'elique,
  du volume d'un domaine d'ad\`eles convenable
  (\cf \'egalement \cite{peyre:torseurs}).
\end{rema}
\subsection{Formule empirique am\'elior\'ee}
\label{subsection.formule}
Dans l'\'etude des espaces de modules de courbes
rationnelles, les courbes tr\`es libres sont caract\'eris\'ees
par le fait que leurs pentes sont strictement positives.
Comme nous allons le voir dans les exem\-ples qui suivent,
dans le contexte arithm\'etique, les points d'une sous-vari\'et\'e
faiblement accumulatrice fix\'ee semblent avoir
une libert\'e qui tend vers~$0$. Toutefois, les points v\'erifiant
la condition $l(x)<\varepsilon$ pour un nombre r\'eel $\varepsilon <1$
peuvent contribuer au terme principal du nombre de points de hauteur
born\'ee, m\^eme dans le cas d'une vari\'et\'e homog\`ene comme
$\PP^1_\QQ\times\PP^1_\QQ$. Cela nous am\`ene \`a consid\'erer
une condition plus faible dans la formule empirique
suivante.
\begin{defis}
  \label{defis.epsilon.libre}
  On notera~$\mathcal D$ l'ensemble
  des applications~$\varepsilon$
  de l'intervalle~${\leftclose 1,+\infty\rightopen}$
  dans~$\leftopen 0,1\rightopen$
  continues et d\'ecroissantes telles que
  \begin{conditions}
  \item L'application $\varepsilon$ tend vers~$0$ en $+\infty$;
  \item Pour tout $\alpha\in\leftopen 0,1]$, l'application
    $t\mapsto \log(t)^\alpha\varepsilon(t)$ tend vers $+\infty$ en $+\infty$.
  \end{conditions}
  Soit~$V$ une belle vari\'et\'e de dimension $n>0$ sur le corps de
  nombres~$\KK$.
  On suppose $V$ munie d'une m\'etrique ad\'elique.
  Soit~$\varepsilon$ une application de~$\mathcal D$. Pour tout
  $B\geq 1$ on d\'efinit l'ensemble
  des points \emph{$\varepsilon$-libres} et de hauteur inf\'erieure \`a~$B$:
  \[V(\KK)^{\elibre}_{H\leq B}=\left\{\,x\in V(\KK)\left|
  H(x)\leq B\text{ et }l(x)\geq \varepsilon(B)
  \,\right.\right\}.\]
  Si un multiple de la classe du faisceau anticanonique $\omega_V^{-1}$
    peut s'\'ecrire comme la somme d'un faisceau ample et d'un diviseur
    \`a croisement normaux stricts, on dira que la vari\'et\'e est
    \emph{vaste}\footnote{Le terme anglais \og big\fg\ signifie une grande
      taille \`a la fois en hauteur et en largeur, le terme \og vaste\fg\
      convient donc pour traduire une condition plus forte que la simple
      grosseur.}.
\end{defis}
\begin{rema}
  Les fonctions $\varepsilon$ envisag\'ee sont du type
  \[t\mapsto\max(1,\log(\log(t)))^{-\alpha}\] pour un r\'eel $\alpha>0$.
\end{rema}
\begin{form}
  \label{formule.empirique}
  Soit~$V$ une belle vari\'et\'e sur~$\KK$ de dimension strictement positive
  et munie d'une m\'etrique ad\'elique. On fait les hypoth\`eses suivantes:
  \begin{conditions}
  \item La vari\'et\'e~$V$ est vaste;
  \item Les groupes de cohomologie $H^i(V,\mathcal O_V)$ sont
    nuls pour $i=1$ ou $2$;
  \item Le groupe de Picard g\'eom\'etrique de~$\overline V$
    est un $\ZZ$-module libre de rang fini;
  \item Les points rationnels de~$V$ sont denses pour la topologie de Zariski;
  \end{conditions}
  On note~$t$ le rang de $\Pic(V)$.
  Pour une application $\varepsilon$ convenable de l'ensemble~$\mathcal D$,
\begin{equation}
  \tag{F} \label{equ.laformule}\card{V(\KK)^{\elibre}_{H\leq B}}\sim
  C(V)B\log(B)^{t-1}
\end{equation}
lorsque $B$ tend vers $+\infty$.
\end{form}
\begin{rema}
  \`A la connaissance de l'auteur, on ne dispose \`a l'heu\-re actuelle d'aucun
  r\'esultat pour une vari\'et\'e~$V$ dont la partie transcendante du
  groupe de Brauer, c'est-\`a-dire l'image de l'extension
  des scalaires $\Br(V)\to\Br(\overline V)$ est non nulle.
  On n'a donc pas d'exemple permettant de confirmer
  que la constante $\beta(V)=H^1(k,\Pic(\overline V))$
  est la bonne dans une telle situation.
\end{rema}

\subsection{La distribution asymptotique}
Comme d\'ej\`a expliqu\'e dans \cite[\S5]{peyre:fano}
la validit\'e de la formule empirique pour tout choix de m\'etrique
implique une \'equidistribution des points rationnels vis-\`a-vis
de la mesure de probabilit\'e \`a densit\'e continue obtenue
en renormalisant la mesure introduite dans la d\'efinition
de la constante empirique.

Pour parler d'\'equidistribution nous allons commencer
par rappeler quelques notions concernant les mesures de comptage.
Pour toute partie finie non vide $X$ de l'espace ad\'elique,
la \emph{mesure de comptage associ\'ee \`a}~$X$
est la mesure d\'efinie sur $V(\Adeles_\KK)$ par
\[\ddelta_X=\frac1{\card X}\sum_{x\in X}\delta_x\,,\]
o\`u $\delta_x$ d\'esigne la mesure de Dirac en~$x$.
Soit $B\geq 1$; lorsque l'ensemble $V(\KK)_{H\leq B}^{\elibre}$
est fini et non vide, on note $\ddelta_{H\leq B}^{\elibre}$
la mesure ainsi associ\'ee \`a  $V(\KK)_{H\leq B}^{\elibre}$.

\begin{defi}
Sur l'espace ad\'elique, sous les hypoth\`eses
(i) \`a (iv) de la formule empirique~\ref{formule.empirique},
qui garantit en particulier que l'espace de Brauer-Manin n'est pas vide,
nous d\'efinissons \'egalement une mesure
de probabilit\'e bor\'elienne support\'ee par l'espace de Brauer-Manin
\`a partir de la mesure $\oomega_V$:
\[\mmu_V(U)=\frac1{\oomega_V(V(\Adeles_\KK)^{\Br})}\oomega_V(U
\cap V(\Adeles_\KK)^{\Br}).\]
\end{defi}
\begin{rema}
  Notons que si $V(\Adeles_\KK)^{\Br}=V(\Adeles_\KK)$, cette mesure
  est le produit des mesures $\mmu_{V,v}=\frac1{\oomega_v(V(\KK_v))}\oomega_v$
  qui est convergent.
  Par le lemme~\ref{lemme.mesure.finie},
  pour presque toute
  place finie~$v$, la mesure de probabilit\'e $\mmu_{V,v}$ est donn\'ee par la
  \emph{mesure de probabilit\'e} d\'efinie par un mod\`ele
  $\mathcal V$ de~$V$ sur $\mathcal O_v$, caract\'eris\'ee par
  la relation
  \begin{equation}
    \oomega_v(\pi_k^{-1}(X))=\frac{\card X}{\card
      {\mathcal V(\mathcal O_v/\mathfrak m_v^k)}}
  \end{equation}
  o\`u $k$ parcourt les entiers entiers strictement positif et
  $X$ les parties de $\mathcal V(\mathcal O_v/\mathfrak m_v^k)$, en notant
  $\pi_k$ la r\'eduction modulo $\mathfrak m_v^k$.
  Cette remarque s'\'etend sans peine \emph{mutatis mutandis} au cas o\`u
  $V(\Adeles_\KK)^{\Br}=W\times\prod_{v\not\in S}V(\KK_v)$,
  pour un ensemble fini $S$ de places et une partie~$W$ ouverte
  et ferm\'ee dans $\prod_{v\in S}V(\KK_v)$.
\end{rema}
\begin{dist}
  \label{distribution.empirique}
  Sous les hypoth\`eses \textup{(i)} \`a \textup{(iv)} de la
  formule empirique, pour toute application~$\varepsilon$ de~$\mathcal D$,
  \begin{equation}
    \tag{E}\label{equ.distribution} \ddelta_{H\leq B}^{\elibre}\longrightarrow
    \mmu_V
  \end{equation}
  au sens faible lorsque~$B$ tend vers $\infty$.
\end{dist}
\begin{rema}
  Cette distribution empirique implique que l'obstruction
  de Brauer-Manin \`a l'approximation faible est la seule.
\end{rema}
\section{Compatibilit\'e avec les exemples}%
\label{section.exemples}
\Subsection{L'espace projectif}
\label{subsection.projectif}
\subsubsection{Minoration de la libert\'e}
Nous allons tout d'abord donner une expression
pour la libert\'e d'un point de l'espace projectif et en d\'eduire que
cette libert\'e est minor\'ee par une constante
strictement positive. Pour l'expression
de la libert\'e, nous allons d'abord fixer une m\'etrique
sur l'espace projectif, qui g\'en\'eralise
l'exemple~\ref{exem.proj.naif}. Soit $n$ un entier strictement positif.
Notons $E=\KK^{n+1}$.
Soit~$w$ une place de~$\KK$. On d\'efinit une norme $\Vert\cdot\Vert_w$
sur $E\otimes_\KK\KK_w$ qu'on peut identifier avec $\KK_w^{n+1}$ par
les formules:
\begin{conditions}
\item $\Vert\boldsymbol y\Vert_w=\sqrt{\sum_{i=0}^ny_i^2}$ si~$w$ est une place
r\'eelle;
\item $\Vert\boldsymbol y\Vert_w=\sum_{i=0}^n|y_i|^2$ si~$w$ est une place
complexe;
\item $\Vert\boldsymbol y\Vert_w=\max_{0\leq i\leq n}(|y_i|_w)$ sinon
\end{conditions}
pour tout $\boldsymbol y=(y_0,\dots,y_n)\in \KK_w^{n+1}$.
Aux places non archim\'ediennes, la norme est donc d\'efinie
par le $\mathcal O_\KK$-module $\mathcal O_\KK^{n+1}$
et
\[\dega(E)=-\,\frac{n+1}2\log|\Delta_\KK|,\] o\`u $\Delta_\KK$ d\'esigne le
discriminant de $\KK$.

Notons $s:\PP(E)\to \Spec(\KK)$ le morphisme structural.
Les normes pr\'ec\'edentes d\'efinissent
une norme ad\'elique sur le fibr\'e vectoriel $s^*(E)$.
En consid\'erant le fibr\'e $\mathcal O_{\PP^n_\KK}(-1)$ comme un sous-fibr\'e
de $s^*(E)$, on peut munir $\mathcal O_{\PP^n_\KK}(-1)$ 
d'une norme ad\'elique. Mais le fibr\'e tangent est canoniquement
isomorphe \`a un quotient de 
$\mathcal O_{\PP^n_\KK}(1)\otimes_{\mathcal O_{\PP^n_\KK}}s^*(E)$,
ce qui permet de le munir de la norme ad\'elique induite.
Si~$F$ est un sous-espace vectoriel de~$E$, il est muni de
la norme ad\'elique induite. 
\begin{prop}\label{prop.6.1}
  Soit~$x$ un point de l'espace projectif~$\PP^n(\KK)$.
  Alors
  \[l(x)=\frac n{n+1}+\min_{F}
  \left(\frac{-n\dega(F)}{\codim_E(F)h(x)}\right),\]
  o\`u $F$ d\'ecrit l'ensemble des sous-espaces vectoriels stricts de $E$
  contenant la
  droite vectorielle correspondant \`a $x$ et $\codim_E(F)=\dim(E)-\dim(F)$.
\end{prop}
\begin{rema}
  \label{rema.dim.un}
  Si $n=1$, l'unique~$F$ qui intervient est la droite vectorielle
  correspondant \`a~$x$ ce qui redonne $l(x)=\frac 12+\frac12=1$,
  formule qui d\'ecoule directement de la d\'efinition de la libert\'e dans
  ce cas.
\end{rema}
\begin{lemm}\label{lemm.6.2}
  Pour tout sous-espace vectoriel~$F$ de~$E$ on a la relation
  \[-\dega(F)\geq\frac{\dim(F)}2\log|\Delta_\KK|.\]
\end{lemm}
\begin{proof}
  Choisissons une partie~$I$ de l'ensemble $\{0,\dots,n\}$ de sorte
  que la projection $(x_0,\dots,x_n)\mapsto (x_i)_{i\in I}$
  d\'efinisse un isomorphisme de $\KK$-espaces vectoriels
  de~$F$ sur~$\KK^{\dim(F)}$. Compte tenu des normes choisies,
  cette projection induit des projections
  orthogonales lorsqu'on tensorise par un compl\'et\'e archim\'edien.
  L'image de l'intersection $M_F=F\cap\mathcal O_\KK^{n+1}$ est contenue
  dans $\mathcal O_\KK^{\dim(F)}$. Donc
  \[\Vol(F_\RR/M_F)\geq \Vol(\KK\otimes_\QQ\RR/\mathcal O_\KK)^{\dim(F)}.
  \qquad\qed\]
  \noqed
\end{proof}
\begin{proof}[D\'emonstration de la proposition~\ref{prop.6.1}]
  Notons~$D$ la droite vectorielle de~$E$ correspondant \`a~$x$.
  L'espace tangent \`a $\PP^n(E)$ en~$x$ est canoniquement isomorphe
  au quotient $D\dual\otimes E/\KK$ o\`u l'injection de $\KK$ dans
  le produit tensoriel $D\dual\otimes E$ est la compos\'ee de l'isomorphisme
  canonique de $\KK$ sur $D\dual\otimes D$ et du plongement de
  $D\dual\otimes D$ dans $D\dual\otimes E$.

  Soit $F'$ un sous-espace vectoriel de $T_x\PP^n(E)$. On note
  $\widetilde F'$ son image inverse dans $D\dual\otimes E$.
  Il existe un unique sous-espace vectoriel~$F$ de~$E$ tel que
  $\widetilde F'$ soit $D\dual\otimes F$. L'espace vectoriel~$F'$
  est donc isomorphe au quotient $D\dual\otimes F/\KK$. Cet isomorphisme
  est compatible avec les normes ad\'eliques choisies,
  il en r\'esulte que
  \[\dega(F')=\dega(F)+\dim(F)\dega(D\dual)=\dega(F)-\dim(F)\dega(D)\]
  D'un autre c\^ot\'e, la hauteur de~$x$ n'est rien d'autre que
  le degr\'e arithm\'etique de l'espace tangent en ce point,
  c'est \`a dire qu'elle est donn\'ee par la formule
  $h(x)=-(n+1)\dega(D)$. Par d\'efinition
  du polyg\^one de Newton, on obtient que le valeur
  de $m_{T_x\PP^n_\KK}(n-1)$ est donn\'ee par
  \begin{align*}
    &\max_{F}\left(\frac{\dega(F)-
      \dim(F)\dega(D)+(n+1)\dega(D)}{n+1-\dim(F)}\right)-(n+1)\dega(D)\\
    =&-n\dega(D)+\max_F\left(\frac{\dega(F)}{\codim(F)}\right),
  \end{align*}
  o\`u $F$ parcourt l'ensemble des sous-espaces vectoriels stricts de~$E$
  contenant~$D$. Par cons\'equent,
  \[\mu_n(T_x\PP^n_\KK)=-\dega(D)+\min_{F}\left(\frac{-\dega(F)}{\codim(F)}
  \right).\]
  Compte tenu du lemme~\ref{lemm.6.2}, cette quantit\'e est positive
  puisque $|\Delta_\KK|\geq 1$.
  La d\'efinition de la libert\'e de~$x$ et l'expression
  pour la hauteur de~$x$ donne les relations
  \[l(x)=\frac{n}{h(x)}\mu_n(T_x(\PP^n_\KK))=\frac n{n+1}+
  \min_{F}
  \left(\frac{-n\dega(F)}{\codim_E(F)h(x)}\right)\]
  comme annonc\'e.
\end{proof}
\begin{corr}
  \label{corr.projectif}
  Avec la norme d\'efinie ci-dessus,
  la libert\'e d'un point rationnel de l'espace projectif de
  dimension~$n$ est minor\'ee par $\frac n{n+1}$.
\end{corr}
\begin{proof}
  Cela r\'esulte de la proposition, du lemme~\ref{lemm.6.2} et du fait que
  le discriminant v\'erifie $|\Delta_\KK|\geq 1$.
\end{proof}
\begin{listrems}
  \remarque
  Ce r\'esultat d\'epend de la m\'etrique ad\'elique choisie; en effet,
  il est possible
  de modifier une m\'etrique ad\'elique de fa\c con \`a donner une
  valeur arbitraire dans $[0,1]$ \`a la libert\'e d'un point fix\'e.
  Toutefois compte-tenu du lemme~\ref{lemm.5.2}, pour tout espace projectif
  $\PP$ de dimension~$n$ muni d'une m\'etrique ad\'elique et tout 
  $\alpha<\frac n{n+1}$, l'ensemble
  \[\{\,x\in\PP(\KK)\mid l(x)<\alpha\,\}\]
  est fini.
  \remarque Notons \'egalement que si l'on choisit un sous-espace vectoriel
  strict~$F$ de~$E$, alors la libert\'e d'un point~$x$
  de~$\PP(F)\subset\PP(E)$ tend
  vers $\frac n{n+1}$ lorsque sa hauteur tend vers $+\infty$.
  Le lemme~\ref{lemm.proj.variante} qui suit montre que ce comportement
  est \og atypique\fg.
\end{listrems}
\begin{corr}
  Pour tout choix de m\'etrique ad\'elique,
  l'espace projectif v\'erifie la formule empirique~\eqref{equ.laformule}
  et la distribution empirique~\eqref{equ.distribution}.
\end{corr}
\begin{proof}
  Rappelons que $\varepsilon(B)$ tend vers~$0$ lorsque~$B$
  tend vers $+\infty$.
  Pour la hauteur choisie dans
  ce paragraphe, compte tenu du corollaire~\ref{corr.projectif},
  il existe un nombre r\'eel $B_0$ tel que
  l'ensemble des points rationnels de l'espace projectif v\'erifiant
  $l(P)\leq \varepsilon(B)$ soit vide pour tout
  $B>B_0$. Pour une hauteur arbitraire,
  lorsque $B>B_0$, il
  r\'esulte de la proposition~\ref{prop.5.2} que tout
  point rationnel de l'espace projectif tel que
  $l(P)\leq \varepsilon(B)$ a une hauteur born\'ee
  par une constante r\'eelle.
  Le r\'esultat d\'ecoule alors
  des propositions~6.1.1 et du corollaire~6.2.17
  de \cite{peyre:fano}, qui se basent en partie sur l'\'etude de
  S.~Schanuel \cite{schanuel:heights}.
\end{proof}
\subsubsection{libert\'e moyenne}
Nous allons maintenant d\'emontrer que
le nombre de points de l'espace projectif dont la libert\'e
v\'erifie $l(x)<1-\eta$
est n\'egligeable devant $B$ et donne donc une contribution n\'egligeable.
\begin{prop}
  \label{lemm.proj.variante}
  Il existe une constante $C>0$ telle que, pour tout $\eta>0$
  on ait la majoration
  \[\card\{\,x\in\PP^n(\KK)\mid H(x)\leq B\text{ et }l(x)<1-\eta\,\}
  <CB^{1-\eta}\]
  pour tout nombre r\'eel $B\geq 1$.
\end{prop}
\begin{proof}
  Par la remarque~\ref{rema.dim.un}, le r\'esultat
  est vrai pour $n=1$. On suppose donc $n\geq 2$.
  Compte tenu de la proposition~\ref{prop.6.1},
  la condition $l(x)<1-\eta$ est \'equivalente
  \`a l'existence d'un sous-espace~$F$ de~$E$ de
  codimension~$c$ qui contient la droite $D$ correspondant \`a~$x$
  et tel que
  \[\frac n{n+1}-\frac{n\dega F}{c\,h(x)}\leq 1-\eta\]
  c'est-\`a-dire
  \begin{equation}
    \label{equ.maj.deg.sous}
    -\dega(F)\leq \frac cn\Bigl(\frac1{n+1}-\eta\Bigr)h(x).
  \end{equation}
  Consid\'erons un instant la grassmannienne $\Gr(n+1-c,E)$
  des sous-espaces de codimension~$c$ dans~$E$. L'espace tangent
  en un point~$F$ est canoniquement isomorphe \`a
  l'espace vectoriel $\Hom(F,E/F)$. On en d\'eduit
  que la hauteur logarithmique
  de $F$, relativement au fibr\'e anticanoniqe
  de la grassmannienne peut \^etre donn\'e par
  \[h(F)=-c\dega(F)+(n+1-c)\dega(E/F)=-(n+1)\dega(F).\]
  D'apr\`es l'estimation de points de hauteur born\'ee sur
  la grassmannienne (\cf \cite{fmt:fano}[\S2]),
  il existe donc une constante~$C$ telle que le nombre
  de sous-espaces~$F$ de codimension~$c$ de~$E$
  tels que $-\dega(F)\leq \log(P)$ soit major\'e par
  $CP^{n+1}$ pour tout $P>1$.

  D'autre part,
  pour un sous-espace~$F$ fix\'e, on reprend la d\'emonstration
  de S.~Schanuel dans~\cite{schanuel:heights}. Pour chaque classe
  d'id\'eaux $\overline{\mathfrak a}$ dans $\mathcal O_\KK$,
  on choisit un id\'eal~$\mathfrak a$ de~$\mathcal O_\KK$
  repr\'esentant
  $\overline{\mathfrak a}$.
  On consid\`ere alors l'ensemble $\mathcal D(F,\mathfrak a,B)$
  des droites $D$ avec $H(D)\leq B$ pour lesquelles la classe
  du $\mathcal O_\KK$-module $D\cap\mathcal O_\KK^{n+1}$ soit
  $\overline{\mathfrak a}$.
  Soit $\Lambda_{\mathfrak a}$ le $\mathcal O_\KK$-module
  $\mathfrak a^{n+1}\cap F$; il forme un r\'eseau du $\RR$-espace
  vectoriel $F\otimes_\QQ\RR$; la longueur minimale d'un vecteur
  de $\Lambda_{\mathfrak a}$ admet une minoration ind\'ependante de~$F$.
 
  Pour toute place archim\'edienne~$w$, notons
  $E_w=E\otimes_\KK\KK_w$. On identifie $E_\RR=E\otimes_\QQ\RR$
  avec $\bigoplus_{w\mid\infty}E_w$. Soit
  \[\llog:\prod_{w\mid\infty}E_w\setminus\{0\}\to\prod_{w\mid\infty}\RR\]
  l'application $(x_w)_{w\mid\infty}\mapsto(\log(\Vert x_w\Vert_w))_{w\mid\infty}$
  et soit $\sigma$ l'application lin\'eaire sur $L=\prod_{w\mid\infty}\RR$
  donn\'ee par $(x_w)_{w\mid\infty}\mapsto\sum_{w\mid\infty}x_w$
  et $\pr$ la projection orthogonale de~$L$ sur $\ker(\sigma)$.
  Rappelons que l'application \[\llog:\prod_{w\mid\infty}\KK_w^*\to L\] donn\'ee
  par $(x_w)_{w\mid\infty}\mapsto(\log(| x_w|_w))_{w\mid\infty}$
  envoie $\mathcal O_\KK^*$ sur un r\'eseau $\Lambda$ de $\ker(\sigma)$.
  On note $\Delta$ un domaine fondamental pour~$\Lambda$ dans $\ker(\sigma)$,
  donn\'e par une base de $\Lambda$.
  On consid\`ere le domaine $\mathcal B$ d\'efinit par la relation
  \[\mathcal B=\{\,y\in E_\RR\setminus\{0\}\mid \pr(\llog(y))
  \in\Delta\text{ et }
  \sigma(\llog(y))\leq 0\,\}\cup\{0\}.\]
  L'ensemble $\llog^{-1}(\pr^{-1}(\Delta))$ est
  le domaine fondamental pour l'action du groupe $\mathcal O_\KK^*$
  modulo les racines de l'unit\'e $\mu_\infty(\KK)$ tel qu'il est
  d\'efini par Schanuel \cite[p.~437]{schanuel:heights}.
  L'ensemble~$\mathcal B$ est invariant par les applications
  de la forme
  \begin{equation}
    \label{equ.isometrie}
    (x_w)_{w\mid\infty}\mapsto(\sigma_w(x_w))_{w\mid\infty}
  \end{equation}
  o\`u $\sigma_w$ est une isom\'etrie de l'espace euclidien $E_w$.

  L'espace $F_\RR=F\otimes_\QQ\RR$, vu comme sous-espace
  de $E_\RR$ est la somme directe des $F_w=F\otimes_\KK\KK_w$
  pour $w\mid\infty$. Il existe donc une famille d'isom\'etries
  $(\sigma_w)_{w\mid\infty}$ telle que l'application donn\'ee
  par~\eqref{equ.isometrie} envoie $F_\RR$ sur ${F_0}_\RR$,
  o\`u $F_0$ est donn\'e par l'annulation des~$c$ derni\`eres
  coordonn\'ees. L'ensemble $\mathcal B_F=F\otimes_\QQ\RR\cap \mathcal B$
  est donc isom\'etrique \`a l'ensemble $\mathcal B_{F_0}$.
  Le domaine $\mathcal B_F$ est donc un domaine born\'e
  dont le bord est une r\'eunion finie
  d'images d'applications de classe $\mathcal C^1$
  et le nombre
  d'applications lipshitziennes intervenant ainsi que les constantes
  de Lipschitz ne d\'epend pas de~$F$.
  
  Le cardinal de l'ensemble $\mathcal D(F,\mathfrak a,B)$
  est major\'e par le nombre de points
  de $\Lambda_\mathfrak a$ dans le domaine dilat\'e
  $T\mathcal B_F$, o\`u~le nombre r\'eel $T$ est d\'efini par
  la relation $B=N(\mathfrak a)^{-n-1}T^{[\KK:\QQ](n+1)}$.
  D'autre part,  le r\'eseau $\Lambda_\mathfrak a$
  est un sous-groupe d'indice fini de $\Lambda_{\mathcal O_\KK}$.
  \`A l'aide de \cite[p.~437, lemma 2]{masservaaler:countingII},
  on obtient donc une majoration du nombre de
  droites dans $\mathcal D(F,\mathfrak a,B)$ de la forme
  \[C\left(\frac{B^{\frac{\dim(F)}{n+1}}}{\exp(-\dega F)}
  +B^{\frac{\dim(F)-1}{n+1}}\right).\]
  \`A l'aide d'une sommation par partie on en d\'eduit
  que le cardinal des droites $D$ avec $H(D)\leq B$
  qui sont contenues dans un sous-espace~$F$ de codimension~$c$
  v\'erifiant~\eqref{equ.maj.deg.sous} est major\'e par
  \begin{multline*}
    C\Bigl(B^{n\frac cn\left(\frac 1{n+1}-\eta\right)+\frac {n+1-c}{n+1}}+
    B^{(n+1)\frac cn\left(\frac 1{n+1}-\eta\right)+\frac{n+1-c-1}{n+1}}\Bigr)\\
    =C(B^{1-c\eta}+B^{1-\frac{n-c}{n(n+1)}-\frac{n+1}nc\eta}).
  \end{multline*}
  La proposition s'obtient en sommant
  cette majoration sur $c\in \{1,\dots,n\}$.
\end{proof}
\begin{rema}
  Cette proposition implique que le cardinal de l'ensemble
  \[\PP^n(\KK)_{\mu_{\max}\leq \log(B)}
  =\{\,x\in\PP^n(\KK)\mid\mu_{\max}(x)\leq\log(B)\,\}\]
  est minor\'e par une expression de la forme $C_\varepsilon B^{n(1-\varepsilon)}$
  pour tout $\varepsilon>0$. En effet, si $l(x)\geq 1-\eta$, alors
  $\mu_{\min}(x)\geq\frac{h(x)}n(1-\eta)$ et donc
  \[\mu_{\max}(x)\leq h(x)-(n-1)\mu_{\min}(x)\leq \frac {h(x)}n(1+(n-1)\eta).\]
  Par cons\'equent, si on pose $\varepsilon=1-(1+(n-1)\eta)^{-1}$,
  les conditions $l(x)\geq 1-\eta$ et $H(x)\leq B^{n(1-\varepsilon)}$
  entra\^\i nent la condition $\mu_{max}(x)\leq \log(B)$.
  Il semble raisonnable d'esp\'erer que le cardinal de cet ensemble
  est en fait \'equivalent \`a une expression de la forme
  $C'(\PP^n_\KK)B^n$.
\end{rema}
\begin{corr}
  La moyenne de la libert\'e, d\'efinie par
  \[\frac 1{\card\PP^n(\QQ)_{H\leq B}}\sum_{x\in\PP^n(\QQ)_{H\leq B}}l(x)\]
  converge vers~$1$ lorsque~$B$ tend vers~$+\infty$.
\end{corr}
\begin{proof}
  Par la proposition pr\'ec\'edente, pour tout $\eta>0$,
  et tout $B\geq 1$,
  \[1\geq \frac 1{\card\PP^n(\QQ)_{H\leq B}}\sum_{x\in\PP^n(\QQ)_{H\leq B}}l(x)
  \geq (1-\eta)(1-CB^{-\eta}).\qquad\qed\]
  \noqed
\end{proof}
\Subsection{Le produit de vari\'et\'es}
\subsubsection{Pr\'eliminaires}
Nous allons commencer par un lemme classique sur les pentes
d'une somme directe.
\begin{lemm}
  \label{lemm.pente.produit}
  Soient $E_1$ et $E_2$ des espaces vectoriels munis de nor\-mes ad\'eliques
  classiques de dimensions respectives $n_1$ et $n_2$.
  On d\'esigne par $\mathcal P(E_1)$ et $\mathcal P(E_2)$ les polyg\^ones
  de Newton correspondants. Alors le polyg\^one de Newton
  de la somme $E_1\oplus E_2$ est
  \[\mathcal P(E_1\oplus E_2)=\mathcal P(E_1)+\mathcal P(E_2).\]
  En particulier, on a la relation
  $\mu_{n_1+n_2}(E_1\oplus E_2)=\min(\mu_{n_1}(E_1),\mu_{n_2}(E_2))$.
\end{lemm}
\begin{proof}
  Soit $F$ un sous-espace vectoriel de $E_1\oplus E_2$.
  Notons $p_1$ la projection de $E_1\oplus E_2$ sur $E_1$. On a une suite
  exacte
  \[0\longrightarrow E_2\cap F\longrightarrow F\longrightarrow p_1(F)
  \longrightarrow 0.\]
  Par d\'efinition de la somme de fibr\'es ad\'eliquement norm\'es (\cf
  l'exemple \ref{exem.fibre.somme}), pour toute place $w$,
  la projection induite $(E_1\oplus E_2)\otimes \KK_w$ sur $E_2\otimes \KK_w$
  est 1-lipschitzienne, et, par cons\'equent,
  $\dega(F/(E_2\cap F))\leq\dega(p_1(F))$ et
  $\dega(F)\leq\dega(E_2\cap F\oplus p_1(F))$.
  \par
  Le polyg\^one $\mathcal P(E_1\oplus E_2)$ est donc l'enveloppe
  convexe de l'ensemble des points de la forme $(\dim(F_1)+\dim(F_2),
  \dega(F_1)+\dega(F_2))$ o\`u $F_i$ est un sous-espace vectoriel de $E_i$
  pour $i\in{1,2}$.
  Cela d\'emontre la premi\`ere assertion. Compte tenu de
  la d\'efinition des fonctions~$m$, on obtient l'\'egalit\'e
  \[m_{E_1\oplus E_2}(n_1+n_2-1)=\max(m_{E_1}(n_1)+\dega(E_2),
  \dega(E_1)+m_{E_2}(n_2))\]
  La seconde relation en d\'ecoule.
\end{proof}
\subsubsection{Un cas particulier}
Avant de passer au cas g\'en\'eral du produit de deux vari\'et\'es,
nous allons traiter le cas particulier d'un produit de droites projectives
qui illustre bien la notion de libert\'e.

\begin{prop}
  Soit $n$ un entier strictement positif. Pour toute application~$\varepsilon$
  de $\mathcal D$, la vari\'et\'e $(\mathbf P^1_\KK)^n$ v\'erifie
  la formule empirique~\ref{formule.empirique}.
\end{prop}
\begin{proof}
  Dans cette preuve,~$V$ d\'esigne $(\mathbf P^1_\KK)^n$.
  Le fibr\'e tangent $TV$ est isomorphe \`a la somme des images inverses des
  fibres $\mathcal O_{\mathbf P^1_\KK}(2)$ et on munit~$V$ de la m\'etrique induite
  par les m\'etriques utilis\'ees pour $\mathbf P^1_\KK$, suivant l'exemple~\ref{exem.fibre.somme}.
  Soit $\boldsymbol x=(x_1,\dots,x_n)$ un point rationnel de~$V$. L'espace
  tangent $T_{\boldsymbol x}V$ est isomorphe \`a $\bigoplus_{i=1}^nT_{x_i}\mathbf P^1_\KK$
  et on d\'eduit du lemme pr\'ec\'edent que $\mu_n(\boldsymbol x)=\min_{1\leq i\leq n}(\mu_1(x_i))$.
  Mais $\mu_1(x_i)=h(x_i)$ ce qui donne la formule
  \[l(\boldsymbol x)=\frac{n\min_{1\leq i\leq n}(h(x_i))}{\sum_{i=1}^nh(x_i)}.\]
  Fixons $B\geq 2$ et $\varepsilon$ un nombre r\'eel strictement positif.
  Il nous faut donc estimer le cardinal de l'ensemble
  \[
  \left\{\,(x_1,\dots,x_n)\in\mathbf P^1(\KK)^n\left|
  \sum_{i=1}^nh(x_i)\leq \min\Bigl(\log(B),\frac n\varepsilon\min_{1\leq i\leq n}(h(x_i))
  \Bigr)\,\right.\right\},\]
  qu'on note $V(\KK)_{H\leq B}^{l\geq\varepsilon}$,
  \`a partir de l'estimation de E.~Landau \cite{landau:elementare}
  \[\card\{\,x\in\mathbf P^1(\KK)\mid h(x)\leq \log(B)\,\}=cB+O(B^{1-\delta})\]
  avec $c=C(\mathbf P^1_\KK)$ et $\delta>0$. On fixe temporairement~$\eta$
  avec $0<\eta<1$.
  Soit $\boldsymbol t=(t_1,\dots,t_n)\in\RRp^n$.
  On \'ecrit $|\boldsymbol t|=\sum_{i=1}^nt_i$. On a donc
  \begin{equation}
    \label{equ.produit.pun}
    \begin{split}
      &\card\left\{\,(x_1,\dots,x_n)\in\mathbf P^1(\KK)^n\left|
      (h(x_i))_{1\leq i\leq n}\in
      \prod_{i=1}^n[t_i,t_i+\eta]\,\right.\right\}\\
      &=c^ne^{|\boldsymbol t|}(e^\eta-1)^n
      +O(e^{|\boldsymbol t|-\delta\min_{1\leq i\leq n}(t_i)})\\
      &=c^ne^{|\boldsymbol t|}\eta^n+O(e^{|\boldsymbol t|}\eta^{n+1})
      +O(e^{|\boldsymbol t|-\delta\min_{1\leq i\leq n}(t_i)}).
    \end{split}
  \end{equation}
  On consid\`ere alors le simplexe compact
  $\Delta_\varepsilon(B)$ de $\RRp^n$ d\'efini par l'in\'egalit\'e
  \[\sum_{i=1}^nt_i\leq\min(\log(B),\frac n\varepsilon\min_{1\leq i\leq n}t_i).\]
  Notons que, si $\boldsymbol t\in\Delta_\varepsilon(B)$,
  le terme d'erreur de~\eqref{equ.produit.pun}
  est major\'e par
  $O(e^{|\boldsymbol t|}\eta^{n+1})+O(e^{(1-\delta\varepsilon/n)|\boldsymbol t|})$.
  Quadrillons maintenant $\RRp^n$ par des cubes de c\^ot\'e~$\eta$;
  le nombre de cubes rencontrant le bord de $\Delta_\varepsilon(B)$
  est major\'e par $O((\log(B)/\eta)^{n-1})$.
  En faisant une comparaison entre somme et int\'egrale on
  obtient donc l'estimation
  \[\card V(\KK)_{H\leq B}^{l\geq\varepsilon}=c^n\int_{\Delta_\varepsilon(B)}
  e^{|\boldsymbol t|}\haar{\boldsymbol t} +O(B\log(B)^n\eta)
  +O\left(\left(\frac{\log(B)}\eta\right)^nB^{1-\delta\varepsilon/n}\right),\]
  les constantes implicites dans les $O$ \'etant ind\'ependantes de
  $\varepsilon$. En prenant $\eta=B^{-\delta\varepsilon/{2n^2}}$ on obtient
  un terme d'erreur en $O(\log(B)^nB^{1-\delta\varepsilon/(2n^2)})$.
  L'int\'egrale vaut $BP_\varepsilon(\log(B))$ o\`u $P_\varepsilon$
  est un polyn\^ome de degr\'e $n-1$ et de coefficient dominant
  $\frac 1{(n-1)!}+O(\varepsilon)$. En utilisant l'\'egalit\'e
  $C(V)=\frac 1{(n-1)!}c^n$, on en d\'eduit la formule
  empirique~\ref{formule.empirique}.
\end{proof}
\begin{rema}
  Il convient de noter que cette d\'emonstration donne
  un excellent contr\^ole du terme d'erreur, mais avec une
  constante diff\'erente, si on
  consid\`ere les points de libert\'e strictement sup\'erieure
  \`a un nombre r\'eel $\varepsilon$ \emph{fix\'e}. En particulier,
  le polyn\^ome en $\log(B)$ dans ce cas ne provient que de l'int\'egrale
  $\int_{\Delta_\varepsilon(B)}e^{|\boldsymbol t|}\haar{\boldsymbol t}$.
  Cette situation est en contraste avec l'\'etude faite par
  S.~Pagelot~\cite{pagelot:nonthese} de $(\PP^1_\QQ)^2_{H\leq B}$.
  En effet, dans ce cas, un calcul direct d\'emontre que chaque
  fibre verticale ou horizontale au dessus d'un point rationnel~$P$ de
  la droite projective a une contribution \'equivalente
  \`a $\frac{C(\PP^1_\KK)}{H(P)}B$ dans l'estimation asymptotique et contribue
  donc au deuxi\`eme terme du polyn\^ome~$B$. Minorer la libert\'e
  des points d\'ecompt\'es par~$\varepsilon$ majore cette contribution par
  \[\frac{C(\PP^1_\KK)}{H(P)}\min(B,H(P)^{2/\varepsilon}).\]
\end{rema}

\subsubsection{Cas g\'en\'eral}
Revenons maintenant au cas g\'en\'eral.
Soient $V_1$ et $V_2$ de belles vari\'et\'es
de dimensions respectives $n_1$ et $n_2$ strictement positives.
On les suppose munies de m\'etriques
ad\'eliques.
Pour $i\in\{1,2\}$, on note $p_i$ la projection de $V_1\times V_2$ sur
$V_i$, $h_i$ la hauteur logarithmique d\'efinie
par la m\'etrique sur $V_i$ et $l_i$ la fonction libert\'e associ\'ee.
Comme le fibr\'e tangent $T(V_1\times V_2)$ est isomorphe
\`a la somme directe $p_1^*TV_1\oplus p_2^*TV_2$, on peut
munir $V_1\times V_2$ de la m\'etriques induite sur la somme directe
(\cf~exemple~\ref{exem.fibre.somme}).
\begin{prop}
  \label{prop.liberte.produit}
  Soient $(x_1,x_2)\in V_1\times V_2(\KK)$. Si $h_1(x_1)$ ou $h_2(x_2)$
  est n\'egatif alors $l(x_1,x_2)=0$. Dans le cas contraire,
  on a la relation
  \[l(x_1,x_2)=(n_1+n_2)
  \frac{\min(l_2(x_2)h_2(x_2)/n_2,l_1(x_1)h_1(x_1)/n_1)}{h_1(x_1)+h_2(x_2)}.\]
\end{prop}
\begin{proof}
  En effet, dans ce cas on a $\mu_{n_i}(T_{x_i}V_i)=h_i(x_i)l_i(x)/n_i$
  pour $i\in\{1,2\}$ et on applique le lemme~\ref{lemm.pente.produit}.
\end{proof}
\begin{rema}
  Comme not\'e dans un cadre plus g\'en\'eral dans la
  remarque \ref{rema.morph.liberte},
  si on fixe un point $x_1$ de $V_1(\KK)$, la liberte
  de $y\in p_1^{-1}(x_1)$ tend vers $0$ quand sa hauteur tend
  vers $+\infty$.
\end{rema}

On pose $n=n_1+n_2$.
Soient $\varepsilon$ et $\varepsilon'$ des applications de
l'ensemble~$\mathcal D$ introduit dans la
d\'efinition~\ref{defis.epsilon.libre}.
On note $\varepsilon'\geq \varepsilon$ si $\varepsilon'(B)\geq\varepsilon(B)$
pour $B\geq 1$.
Dans le th\'eor\`eme qui suit, on d\'esigne par
$\varepsilon$ une application
de l'ensemble $\mathcal D$ qui v\'erifie, en plus des
conditions (i) et (ii) de la d\'efinition~\ref{defis.epsilon.libre},
la condition suivante:
\begin{conditions}
  \setcounter{enumi}{2}
\item L'application
    $t\mapsto \log(t)^{1/2}\varepsilon(t)$ est croissante.
\end{conditions}

\begin{theo}
  \label{theo.produit}%
  On suppose que les vari\'et\'es~$V_1$ et~$V_2$ v\'erifient
  les conditions \textup{(i)} \`a \textup{(iv)}
  de la formule empirique et que, pour $i\in\{1,2\}$ et
  toute application~$\varepsilon'\in\mathcal D$ telle que
  $\varepsilon'\geq\frac {n_i^2}{2n^2}\varepsilon^2$,
  on ait l'estimation~\eqref{equ.laformule}:
  \[\card V_i(\KK)_{H_i\leq B}^{\elibre[']}=C(V_i)B\log(B)^{t_i-1}
  (1+o_{\varepsilon'}(B)),\]
  avec $t_i=\rg(\Pic(V_i))$.
  Alors on a l'expression
  \[\card (V_1\times V_2)(\KK)_{H\leq B}^{\elibre}=C(V_1\times V_2)
  B\log(B)^{t_1+t_2-1}(1+o_\varepsilon(B)).\]
\end{theo}
\begin{proof}
  Quitte \`a remplacer $\varepsilon$ par
  l'application donn\'ee par
  ${t\mapsto\min(\frac12,\varepsilon(t))}$, nous
  pouvons, sans perte de g\'en\'eralit\'e,
  supposer qu'on a l'in\'egalit\'e $\varepsilon(1)\leq\frac 12$.
  Soit $i\in\{1,2\}$.
  \'Etant donn\'e une application $\varepsilon'$ comme dans
  l'\'enonc\'e du th\'eor\`eme, l'application de
  $\leftclose 1,+\infty\rightopen$ dans $\NN$ d\'efinie par
  $B\mapsto V_i(\KK)_{H_i\leq B}^{\elibre[']}$ est croissante.
  La condition $l_i(x)\geq\varepsilon'(B)$ impose l'in\'egalit\'e $H_i(x)>1$,
  si bien qu'elle vaut~$0$ en~$1$. En outre, elle est \'egale
  \`a sa limite \`a droite en tout point, donc nulle sur un voisinage
  de $1$.
  On peut donc poser, pour $t\geq 1$,
  \[\eta_{i,\varepsilon'}(t)=\sup_{B> t}\left|
  \frac{V_i(\KK)_{H_i\leq B}^{\elibre[']}}{C(V_i)B\log(B)^{t_i-1}}-1\right|.\]
  Cela d\'efinit une application d\'ecroissante qui converge
  vers~$0$ lorsque~$B$ tend vers $+\infty$.
  \noqed
\end{proof}
\begin{rema}
  \label{rema.produit.multi}
  Si $S$ est un ensemble \emph{fini} d'applications $\varepsilon'$ comme ci-dessus.
  L'application d\'efinie par la relation
  $\eta_{i,S}(t)=\max_{\varepsilon'\in S}\eta_{i,\varepsilon}$
  jouit de propri\'et\'es analogues.
\end{rema}
 
  Soit $(x_1,x_2)$ un point rationnel du produit. Par la
  proposition~\ref{prop.liberte.produit}, la condition
  $l(x_1,x_2)\geq\varepsilon(B)$ implique tout d'abord
  que $h_1(x_1)>0$ et $h_2(x_2)>0$.
  Par suite, elle implique \'egalement les in\'egalit\'es
  \[l_1(x_1)\geq \frac{n_1}n\varepsilon(B)\qquad\text{et}\qquad
  l_2(x_2)\geq \frac{n_2}n\varepsilon(B).\]
  On note dans la suite $\varepsilon_1=\frac{n_1}n\varepsilon$
  et $\varepsilon_2=\frac{n_2}n\varepsilon$.
  Nous allons d\'ecouper la preuve en une s\'erie de lemmes.
\begin{lemm}
  \label{lemm.produit.epsilon}
  Soit $\lambda\in\leftopen 0,1\rightopen$
  et soit $\varepsilon'=\lambda\varepsilon$.
  Soit $B>1$, pour tout $P\in [B^{1/2},B]$, on a les in\'egalit\'es
  \[\varepsilon'(B)\leq\varepsilon'(P)\leq2\varepsilon'(B),\]
  et, pour tout $P\in[B^{\varepsilon'(B)},B]$, on a
  \[\varepsilon'(B)\leq\varepsilon'(P)\leq\sqrt{\varepsilon'(B)}.\]
\end{lemm}
\begin{proof}
  Les in\'egalit\'es de gauche r\'esultent du fait que $\varepsilon$
  est suppos\'ee d\'ecroissante. Comme $t\mapsto\log(t)\varepsilon'(t)$
  est croissante, on a
  \[\log(B^{1/2})\varepsilon'(B^{1/2})\leq
  \log(B)\varepsilon'(B)\] ce qui donne la premi\`ere majoration.
  De m\^eme, l'in\'egalit\'e
  \[\sqrt{\log(B^{\varepsilon'(B)})}\varepsilon'(B^{\varepsilon'(B)})\leq
  \sqrt{\log(B)}\varepsilon'(B)\]
  permet de prouver la seconde.
\end{proof}
\begin{lemm}
  \label{lemm.maj.carre}
  Le cardinal de l'ensemble des $(x_1,x_2)\in V_1\times V_2(\KK)$
  v\'erifiant les conditions
  \[H_1(x_1)\leq B^{1/2},\quad H_2(x_2)\leq B^{1/2},
  \quad l_1(x_1)\geq \varepsilon_1(B)\quad
  \text{et}\quad l_2(x_2)\geq\varepsilon_2(B)\]
  est, \`a une constante pr\`es, major\'e par $B\log(B)^{t_1+t_2-2}$.
\end{lemm}
\begin{proof}
  Le lemme pr\'ec\'edent permet de majorer le cardinal consid\'er\'e
  par celui de l'ensemble des $(x_1,x_2)\in V_1\times V_2(\KK)$
  tels que
  \[H_1(x_1)\leq B^{1/2},\  H_2(x_2)\leq B^{1/2},\  l_1(x_1)\geq
  \frac12\varepsilon_1(B^{1/2})\ 
  \text{et}\  l_2(x_2)\geq\frac12
  \varepsilon_2(B^{1/2})\]
  et on applique les hypoth\`eses du th\'eor\`eme \`a $\frac 12\varepsilon_1$
  et $\frac 12\varepsilon_2$.
\end{proof}
Par sym\'etrie, compte tenu du dernier lemme, il  nous suffit
d'estimer le cardinal de l'ensemble des $(x_1,x_2)\in V_1\times V_2(\KK)$
tels que
\[H_1(x_1)\geq B^{1/2},\quad H_1(x_1)H_2(x_2)\leq B,\quad\text{et}\quad
l(x_1,x_2)\geq\varepsilon(B).\]
La condition sur la libert\'e est la conjonction des deux
conditions suivantes:
\begin{conditions}
\item $\displaystyle l_1(x_1)
  \geq \frac{h_1(x_1)+h_2(x_2)}{h_1(x_1)}\varepsilon_1(B)$;
\item $\displaystyle l_2(x_2)
  \geq \frac{h_1(x_1)+h_2(x_2)}{h_2(x_2)}\varepsilon_2(B)$.
\end{conditions}
Sous les hypoth\`eses $H_1(x_1)\geq B^{1/2}$ et $H_1(x_1)H_2(x_2)\leq B$,
on a
\[\frac{h_1(x_1)+h_2(x_2)}{h_1(x_1)}\in [1,2]
\quad\text{et}\quad\frac{h_1(x_1)+h_2(x_2)}{h_2(x_2)}
\in\left[\frac{\log(B)}{2h_2(x_2)}+1,\frac{\log(B)}{h_2(x_2)}\right].\]
On introduit les quatres conditions suivantes
\begin{align*}
  \text{\textup{(i${}^+_B$)}}&\ l_1(x_1)\geq \varepsilon_1(B),\qquad&
  \text{\textup{(i${}^-_B$)}}&\ l_1(x_1)\geq 2\varepsilon_1(B),\\
  \text{\textup{(ii${}^+_B$)}}&\ l_2(x_2)\geq
  \left(\frac{\log(B)}{2h_2(x_2)}+1\right)
  \varepsilon_2(B),\qquad&
  \text{\textup{(ii${}^-_B$)}}&\ l_2(x_2)\geq
  \frac{\log(B)}{h_2(x_2)}\varepsilon_2(B).
\end{align*}
Comme $l_2(x_2)\leq 1$, la condition (ii${}^+_B$) implique
que $H_2(x_2)\geq B^{\frac{\epsilon_2(B)}2}$.
Consid\'erons maintenant les conditions
\begin{conditions}
\item[(iii${}^+_B$)]$l_2(x_2)\geq\epsilon_2(B)$ et $H_2(x_2)
  \geq B^{\frac{\varepsilon_2(B)}2}$;
\item[(iii${}^-_B$)]$l_2(x_2)\geq\sqrt{\epsilon_2(B)}$ et $H_2(x_2)\geq B^{\sqrt{\varepsilon_2(B)}}$.
\end{conditions}
Alors, sous les hypoth\`eses pr\'ec\'edentes, on a les implications
\[\text{\textup{(i${}^-_B$)}}\Longrightarrow\text{\textup{(i)}}
\Longrightarrow\text{\textup{(i${}^+_B$)}}\]
et
\[\text{\textup{(iii${}^-_B$)}}\Longrightarrow
\text{\textup{(ii${}^-_B$)}}\Longrightarrow\text{\textup{(ii)}}
\Longrightarrow\text{\textup{(ii${}^+_B$)}}\Longrightarrow
\text{\textup{(iii${}^+_B$)}}.\]
Pour $P\in [B^{\frac12},B]$, on note $N_1^+(P,B)$ (\resp\ $N_1^-(P,B)$)
le cardinal des $x_1\in V_1(\KK)$ tels que $H_1(x_1)\leq P$
et qui v\'erifient la condition (i${}^+_B$) (\resp\ (i${}^-_B$)).
On note $\mathcal E^+_2(B)$ (\resp\ $\mathcal E^-_2(B)$) l'ensemble
des $x_2\in V_2(\KK)$ tels que $H_2(x_2)\leq B^{\frac12}$
et qui v\'erifient la condition (iii${}^+_B$) (\resp\ (iii${}^-_B$)).
Le nombre qui nous int\'eresse est minor\'e (\resp. major\'e)
par
\[\sum_{x_2\in\mathcal E_2^-(B)}\left(N_1^-\left(\frac B{H_2(x_2)},B\right)
-N_1^-(B^{\frac12},B)\right),\]
(\resp\ par
\[\left.\sum_{x_2\in\mathcal E_2^+(B)}\left(N_1^+\left(\frac B{H_2(x_2)},B\right)
-N_1^+(B^{\frac12},B)\right)\right).\]
En appliquant une nouvelle fois le lemme~\ref{lemm.maj.carre},
nous constatons que la contribution de
$\sum_{x\in\mathcal E_2^+}N_1^+(B^{\frac12},B)$ est n\'egligeable
devant $B\log(B)^{t_1+t_2-1}$.
Il nous reste \`a estimer la somme pour le terme principal.
\begin{lemm}
  Il existe une application $\eta'_1:\leftclose 1,+\infty\rightopen\to\RR$
  d\'ecroissante et tendant vers~$0$ en $+\infty$ telle
  que
  \[\left|\frac{N_1(P,B)}{C_1(V)P\log(P)^{t_1-1}}-1\right|\leq \eta'_1(B)\]
  pour $N_1\in\{N_1^+,N_1^-\}$, $B> 1$ et $P\in [B^{\frac 12},B]$.
\end{lemm}
\begin{proof}
  Comme $P\in [B^{\frac12},B]$, le lemme~\ref{lemm.produit.epsilon}
  donne les in\'egalit\'es
  $\frac 12\varepsilon_1(P)\leq\varepsilon_1(B)\leq\varepsilon_1(P)$.
  En appliquant la remarque~\ref{rema.produit.multi} \`a
  l'ensemble $S=\{\frac 12\varepsilon_1,\varepsilon_1,2\varepsilon_1\}$
  on obtient une application d\'ecroissante $\eta_1$ de sorte
  que les termes d'erreur du lemme soient major\'es par $\eta_1(P)$.
  L'application $\eta'_1:B\mapsto\eta_1(B^{1/2})$ satisfait alors
  la conclusion du lemme.
\end{proof}
Ce lemme nous permet donc de nous ramener \`a estimer
\[\sum_{x_2\in\mathcal E_2(B)}\frac{1}{H_2(x_2)}\log
\left(\frac B{H_2(x_2)}\right)^{t_!-1}\]
pour $\mathcal E_2\in\{\mathcal E_2^+,\mathcal E_2^-\}$.
On consid\`ere l'application $f_B:[1,B^{\frac12}]\to\RR$
d\'efinie par ${P\mapsto \frac 1P(\log(B)-\log(P))^{t_1-1}}$.
On note $g_B^+$ (\resp\ $g_B^-$) l'application qui \`a $P\in[1,B^{\frac 12}]$
associe le cardinal des $x\in V_2(\KK)$ tels que $H_2(x)\leq P$
et $l_2(x)\geq \varepsilon_2(B)$ (\resp\ $l_2(x)\geq\sqrt{\varepsilon_2(B)}$).
En utilisant les notations
des int\'egrales de Stieltjes (\cf \cite[\S I.0.1]{tenenbaum:analytique}),
les sommes ci-dessus se mettent donc sous la forme
\begin{equation}
  \label{equ.parpartie}
  \int_{B^{\epsilon'(B)}}^{B^{\frac12}}f_B(u)\haar{g_B(u)}=
      [f_B(u)g_B(u)]_{B^{\epsilon'(B)}}^{B^{\frac12}}-
      \int_{B^{\epsilon'(B)}}^{B^{\frac12}}f'_B(u)
      g_B(u)\haar u,
\end{equation}
o\`u $\varepsilon'\in\{\epsilon_2,\sqrt{\epsilon_2}\}$ et
$g_B\in\{g^+_B,g^-_B\}$; l'\'egalit\'e vient de la formule
d'Abel qui s'\'ecrit ici comme une int\'egration par partie.
\begin{lemm}
  \label{lemm.produit.deuxieme}
  Il existe une application $\eta'_2:\leftclose 1,+\infty\rightopen\to\RR$
  d\'ecroissante et tendant vers~$0$ en $+\infty$ telle
  que
  \[\left|\frac{g_B(P)}{C_2(V)P\log(P)^{t_2-1}}-1\right|\leq \eta'_2(B)\]
  pour $B> 1$, $P\in [B^{\varepsilon_2(B)},B^{\frac 12}]$ et
  $g_B\in\{g_B^+,g_B^-\}$.
\end{lemm}
\begin{proof}
  Comme $P\in [B^{\varepsilon_2(B)},B^{\frac 12}]$,
  le lemme~\ref{lemm.produit.epsilon}
  donne les in\'egalit\'es
  $\varepsilon_2(P)^2\leq\varepsilon_2(B)\leq\varepsilon_2(P)$.
  En appliquant la remarque~\ref{rema.produit.multi} \`a
  l'ensemble $S=\{\varepsilon_2^2,\varepsilon_2,\sqrt\varepsilon_2\}$,
  on obtient une application d\'ecroissante $\eta_2$ de sorte
  que les termes d'erreur du lemme soient major\'es par $\eta_2(P)$.
  L'application $\eta'_2:B\mapsto\eta_2(B^{\epsilon_2(B)})$ satisfait alors
  la conclusion du lemme puisque l'application
  $B\mapsto\log(B)\varepsilon_2(B)$ est croissante
  et tend vers $+\infty$ en $+\infty$.
\end{proof}
\begin{proof}[Fin de la preuve du th\'eor\`eme~\ref{theo.produit}]
  La d\'eriv\'ee de l'application $f_B$ est donn\'e par
  $f_B'(t)=-1/t^2$ si $t_1=1$ et par
  \[f_B'(t)=\frac{-1}{t^2}((\log(B)-\log(t))^{t_1-1}
  +(t_1-1)(\log(B)-\log(t))^{t_1-2})\]
  si $t_1\geq 2$.
  Il r\'esulte du lemme~\ref{lemm.produit.deuxieme} que
  \[
  \begin{split}
    [f_B(u)g_B(u)]_{B^{\varepsilon'(B)}}^{B^{\frac 12}}&\leq f_B(B^{\frac12})
    g_B(B^{\frac12})\\
    &\leq C(V_2)\log(B^{\frac12})^{t_1-1}\log(B^{\frac12})^{t_2-1}
    (1+\eta'_2(B))\\
    &\leq C(V_2)\log(B)^{t_1+t_2-2}
    (1+\eta'_2(B)).
  \end{split}
  \]
  Ce terme est donc n\'egligeable devant $\log(B)^{t_1+t_2-1}$.
  D'autre part, pour $0<\lambda<\mu<1$ et
  pour un entier $m\geq 1$, on a les \'egalit\'es
  \[\begin{split}
  &\int_{B^\lambda}^{B^\mu}\frac 1{u^2}(\log(B)-\log(u))^{m-1}
  u\log(u)^{t_2-1}\haar u\\
  &\quad=\log(B)^{t_2+m-1}\int_{B^\lambda}^{B^\mu}
  \left(1-\frac{\log(u)}{\log(B)}\right)^{m-1}
  \left(\frac{\log(u)}{\log(B)}\right)^{t_2-1}
  \haar{\left(\frac{\log(u)}{\log(B)}\right)}\\
  &\quad=\log(B)^{t_2+m-1}\int_\lambda^\mu (1-u)^{m-1}u^{t_2-1}\haar u.
  \end{split}\]
  Comme les applications $\varepsilon_2$ et $\sqrt{\varepsilon_2}$
  convergent vers~$0$ en~$+\infty$, le terme de droite
  de~\eqref{equ.parpartie} est \'equivalent \`a
  \[C(V_2)\log(B)^{t_1+t_2-1}\int_0^{\frac12}(1-u)^{t_1-1}u^{t_2-1}\haar u.\]
  Le cardinal de l'ensemble des $(x_1,x_2)\in V_1\times V_2(\KK)$
  tels que
  \[H_1(x_1)\geq B^{1/2},\quad H_1(x_1)H_2(x_2)\leq B,\quad\text{et}\quad
  l(x_1,x_2)\geq\varepsilon(B)\]
  est donc \'equivalent \`a
  \[C(V_1)C(V_2)B\log(B)^{t_1+t_2-1}\int_0^{\frac12}(1-u)^{t_1-1}u^{t_2-1}\haar u.\]
  Par sym\'etrie, le cardinal de l'ensemble des
  $(x_1,x_2)\in V_1\times V_2(\KK)$
  tels que
  \[H_2(x_2)\geq B^{1/2},\quad H_1(x_1)H_2(x_2)\leq B,\quad\text{et}\quad
  l(x_1,x_2)\geq\varepsilon(B)\]
  est \'equivalent \`a
  \[C(V_1)C(V_2)B\log(B)^{t_1+t_2-1}\int_{\frac12}^1(1-u)^{t_1-1}u^{t_2-1}\haar u.\]
  Mais compte tenu des propri\'et\'es de la fonction b\^eta,
  on a l'\'egalit\'e
  \[\int_0^1(1-u)^{t_1-1}u^{t_2-1}\haar u=\frac{(t_1-1)!(t_2-1)!}{(t_1+t_2-1)!}.\]
  Il nous reste pour conclure \`a rappeler la formule suivante
  (cf. \cite[proposition~4.1]{peyre:fano}):
  \[C(V_1\times V_2)=\frac{(t_1-1)!(t_2-1)!}{(t_1+t_2-1)!}C(V_1)C(V_2).\qed\]
  \noqed
\end{proof}
\begin{listrems}
  \remarque
  Le terme d'erreur donn\'e par cette d\'emonstration est explicite
  mais assez pitoyable, en particulier
  si on le compare avec celui obtenu
  dans le cas particulier d'un produit de droites projectives.
  Nous manquons actuellement d'exemples pour savoir quel r\'esultat
  optimal pourrait \^etre attendu ici.

  \remarque
  Notons que l'ensemble d\'efini par
  \[V(K)_{\mu_{\max}\leq B}=\{\,x\in V(\KK)\mid \mu_{\max}(x)\leq \log(B)\,\}.\]
  aurait l'avantage de se comporter
  mieux avec le produit de vari\'et\'es puisque
  l'ensemble ${V_1\times V_2(\KK)_{\mu_{\max}\leq \log(B)}}$
  est tout simplement le produit des ensembles
  $V_i(\KK)_{\mu_{\max}\leq \log(B)}$.
\end{listrems}

\subsubsection{\'Equidistribution}
L'\'equidistribution au sens de la distribution
empirique~\ref{distribution.empirique}, est \'egalement stable par produit.
On conserve les notations pr\'ec\'edant le th\'eor\`eme~\ref{theo.produit}.
\begin{theo}
  \label{theo.produit.equidist}%
  On suppose que les vari\'et\'es~$V_1$ et~$V_2$ v\'erifient
  les conditions \textup{(i)} \`a \textup{(iv)}
  de la formule empirique et que, pour $i\in\{1,2\}$ et
  toute application~$\varepsilon'\in\mathcal D$ telle que
  $\varepsilon'\geq\frac {n_i^2}{2n^2}\varepsilon^2$
  la vari\'et\'e $V_i$ v\'erifie la formule empirique~\eqref{equ.laformule}
  et la distribution~\eqref{equ.distribution}.
  Alors $V_1\times V_2$ v\'erifie \'egalement~\eqref{equ.distribution}.
\end{theo}
\begin{proof}
  Compte tenu de \cite[\S3]{peyre:fano} et du th\'eor\`eme
  pr\'ec\'edent, il suffit de v\'erifier que,
  pour tout bon ouvert $W\subset V_1\times V_2(\Adeles_\KK)$,
  c'est-\`a-dire tout ouvert dont le bord $\partial W$
  est de mesure nulle pour la mesure ad\'elique,
  le cardinal de l'ensemble $(V_1\times V_2(\KK)\cap W)_{H\leq B}^{\elibre}$
  est \'equivalent \`a
  \[\alpha(V_1\times V_2)\beta(V_1\times V_2)
  \oomega_{V_1\times V_2}(W\cap V(\Adeles_\KK)^{\Br})B\log(B)^{t_!+t_2-1}.\]
  En outre, il suffit de le d\'emontrer pour les ouverts qui sont
  de la forme $W_1\times W_2$
  o\`u~$W_1$ (\resp~$W_2$) est un bon ouvert de~$V_1$ (\resp~$V_2$).
  La d\'emonstration s'obtient alors en rempla\c cant
  simplement $V_i(\KK)$
  par $V_i(\KK)\cap W_i$ pour $i\in\{1,2\}$ dans la d\'emonstration du
  th\'eor\`eme~\ref{theo.produit}.
\end{proof}
\subsection{Compatibilit\'e avec la m\'ethode du cercle}

Il est connu que les r\'esultats de la m\'ethode du cercle sont compatibles
avec la version initiale du principe de Batyrev et Manin
(\cite[\S1.4]{fmt:fano} et \cite[corollaire~5.4.9]{peyre:fano})
en prenant comme ouvert de Zariski la vari\'et\'e elle-m\^eme.
La difficult\'e ici, comme dans le cas de l'espace projectif, est
donc de d\'emontrer que les points de petite libert\'e et de hauteur
born\'ee donnent une contribution n\'egligeable.

Soit~$V$ une intersection compl\`ete
lisse de~$m$ hypersurfaces de degr\'es respectifs $d_1,\dots,d_m$
dans l'espace projectif $\PP^N_\KK$ avec $d_i\geq 2$ pour~$i$ appartenant
\`a $\{1,\dots,m\}$.
On pose $|\boldsymbol d|=\sum_{i=1}^md_i$.
On suppose que la dimension~$n=N-m$
de~$V$ v\'erifie $n\geq 3$.
Rappelons que, dans ce cas, le torseur universel de~$V$
s'identifie au c\^one \'epoint\'e~$W\subset\Aff^{N+1}_\KK\setminus\{0\}$
au-dessus de~$V$. On note~$\pi$ la projection de~$W$ vers~$V$.
On munit l'espace projectif de la m\^eme m\'etrique ad\'elique
qu'au paragraphe~\ref{subsection.projectif}. Le fibr\'e tangent
de~$V$ est alors un sous-fibr\'e de l'image inverse de $T\PP^N_\KK$
sur~$V$ et on le munit de la m\'etrique ad\'elique induite.
\begin{prop}
  \label{prop.cercle.liberte}
  Soit $x\in V(\KK)$ et soit $y\in\pi^{-1}(x)$. Alors la libert\'e
  $l(x)$ est donn\'ee par l'expression
  \begin{equation*}
    \frac n{h(x)}\max\left(0,
    \min_F\left(\frac{\dega(T_yW)-\dega(F)}{n+1-\dim(F)}\right)
    -\dega(D)\right)
  \end{equation*}
  qui, \`a un terme $O(\frac 1{h(x)})$ pr\`es, peut s'\'ecrire
  \begin{equation*}
    \frac{n}{N+1-|\boldsymbol d|}\max\left(0,1+\min_{F}\left(
      \frac{m-|\boldsymbol d|-
        (N+1-|\boldsymbol d|)\dega(F)/h(x)}{n+1-\dim(F)}\right)\right)
  \end{equation*}
  o\`u $F$ d\'ecrit l'ensemble des sous-espaces stricts
  de $T_yW$ contenant $y$.
\end{prop}
\begin{rema}
  Par la remarque~\ref{rema.espaces.minima}, \`a un terme en $O(\frac 1{h(x)})$
  pr\`es, il suffit de consid\'erer les hyperplans~$F$ de $T_yW$. La condition
  $l(x)<\varepsilon(B)$ se traduit donc essentiellement
  par l'existence d'un hyperplan
  $F\subset T_yW$ tel que
  \[-\dega(F)\leq\left(\frac{\varepsilon(B)}n+
  \frac{|\boldsymbol d|-1-n}{N+1-|\boldsymbol d|}\right)h(x).\]
  Cette condition peut \^etre vue comme l'existence d'un \og petit\fg\
  vecteur dans le r\'eseau dual du r\'eseau $T_yW\cap \ZZ^{N+1}$.
\end{rema}
\begin{proof}
  Notons $E=\KK^{n+1}$ et~$D$ la droite de l'espace vectoriel~$E$ correspondant au point~$x$.
  L'espace tangent $T_xV$ s'identifie alors au quotient
  de l'espace ${D\dual\otimes T_yW}$
  par $D\dual\otimes D$. Soit $F'$ un sous-espace de $T_xV$.
  Il existe un unique sous-espace~$F$ de $T_yW$ contenant~$D$
  tel que~$F'$ corresponde par l'identification pr\'ec\'edente au quotient
  de $D\dual\otimes F$ par $D\dual\otimes D$. Par cons\'equent
  \[\dega(F')=\dega(F)+\dim(F)\dega(D\dual)=\dega(F)-\dim(F)\dega(D).\]
  Or $h(x)$ vaut $\dega(T_xV)=\dega(T_yW)-(n+1)\dega(D)$.
  Par cons\'equent, la pente minimale $\mu_n(T_xV)$ est donn\'ee par la formule
  \[\begin{split}&\min_F\left(\frac{\dega(T_yW)-(n+1)\dega(D)
   -(\dega(F)-\dim(F)\dega(D))}{n-(\dim(F)-1)}\right)\\
  &=\min_F\left(\frac{\dega(T_yW)-\dega(F)}{n+1-\dim(F)}\right)
  -\dega(D).
  \end{split}\]
  En divisant par $h(x)/n$, on obtient l'expression de la libert\'e.
  Pour la deuxi\`eme expression, il suffit de remarquer
  que  \[|h(x)+(N+1-|\boldsymbol d|)\dega(D)|\] est major\'e par une constante
  puisque le fibr\'e anticanonique est isomorphe \`a
  $\mathcal O_V(N+1-|\boldsymbol d|)$.
\end{proof}
\begin{rema}
  Contrairement au cas de l'espace projectif,
  on ne peut \'evidemment pas esp\'erer, en g\'en\'eral,
  trouver de minimum absolu strictement
  positif pour la libert\'e.
  Par exemple, $\PP^1\times\PP^1$ se r\'ealise comme une quadrique d\'eploy\'ee
  et on a vu que la borne inf\'erieure de la libert\'e dans ce cas
  est nulle. De m\^eme, si on consid\`ere une surface cubique,
  les points sur les $27$ droites de la surface ont forc\'ement
  une libert\'e nulle \`a l'exception pr\`es d'un nombre fini d'entre eux.
\end{rema}
\begin{prop}
  Soit
  $Q$ une quadrique projective lisse de dimension $n\geq 3$
  sur~$\QQ$,
  Alors pour tout application~$\varepsilon$ de l'ensemble~$\mathcal D$,
  on a l'\'equivalence
  \[\card Q(\QQ)_{H\leq B}^{\elibre}\sim C(Q)B\]
  quand $B$ tend vers $+\infty$.
\end{prop}
\begin{proof}
  Compte tenu de \cite{peyre:fano}[corollaire 5.4.9] qui
  repose sur le r\'esultat tr\`es g\'en\'eral
  de Birch~\cite{birch:forms},
  il suffit de d\'emontrer que dans ce cas particulier,
  le nombre de points~$x$ de la quadrique v\'erifiant $H(x)\leq B$
  et $l(x)<\varepsilon(B)$ sont en nombre n\'egligeable.
  Dans ce cas particulier, 
  la formule de la proposition~\ref{prop.cercle.liberte}
  donne
  \[l(x)\geq 1+\min_F\left(\frac{-1-n\dega(F)/h(x)}{n+1-\dim(F)}\right)
    +O\left(\frac 1{h(x)}\right).\]
  Soit $W\subset \Aff_\QQ^{n+1}\setminus\{0\}$ le c\^one \'epoint\'e au-dessus
  de la quadrique~$Q$.
  Soit $y$ un repr\'esentant de~$x$ dans $\QQ^{n+1}$.
  La condition $l(x)<\varepsilon(B)$ implique donc l'existence
  d'un sous-espace vectoriel $F$ de $T_yW$
  de codimension $1$ dans cet espace tel que
  \[-\dega(F)\leq C+\frac 1n\epsilon(B)\log(B).\]
  En reprenant le raisonnement fait pour d\'emontrer la
  proposition~\ref{lemm.proj.variante}, le nombre de tels sous-espaces
  est major\'e \`a une constante pr\`es
  par $B^{\frac{n+2}n\varepsilon(B)}$ et chacun d'entre eux contient
  au plus $C'B^{\frac{n}{n+1}}$ points~$x$ avec $H(x)<B$, ce qui
  permet de conclure.
\end{proof}
\begin{rema}
  Malheureusement, cette preuve ne s'\'etend pas directement
  au cas du degr\'e $d\geq 3$.
\end{rema}
\section{Compatibilit\'e avec les contre-exemples}
L'objectif de cette partie de l'article est de v\'erifier que
la condition sur la libert\'e d\'etecte bien les mauvais points.
Nous allons donc passer en revue un certain nombre de contre-exemples
connus \`a la question initiale de Batyrev et Manin et analyser la libert\'e
des points rationnels des vari\'et\'es faiblement accumulatrices dans
chacun de ces cas.

\subsection{Rappels sur les parties accumulatrices}
Commen\c cons par rappeler quelques notions concernant
les sous-ensembles accumulateurs.
Pour les en\-sem\-bles minces,
on \'etend la d\'efinition de Serre \cite[\S9.1]{serre:mordellweil}
de la fa\c con suivante:
\begin{defi}
  Soit~$V$ une bonne vari\'et\'e sur le corps de nombres~$\KK$,
  une partie \emph{mince}
  de $V(\KK)$ est une partie~$M$ telle qu'il existe un
  morphismes de vari\'et\'es $\pi:X\to V$ qui v\'erifie
  les deux conditions suivantes:
  \begin{conditions}
  \item La partie~$M$ est contenue dans l'image $\pi(X(\KK))$;
  \item La fibre de~$\pi$ au point g\'en\'erique est finie et
    l'application~$\pi$ n'a pas de section rationnelle.
  \end{conditions}
\end{defi}
\begin{rema}
  Il convient de noter
  qu'avec cette d\'efinition, l'ensemble des points rationnels
  d'une courbe elliptique est mince. En effet en choisissant un
  syst\`eme de repr\'esentants $(P_1,\dots,P_k)$
  du quotient fini $E(\KK)/2E(\KK)$,
  on d\'efinit l'application $\pi:\coprod_{i=1}^kE\to E$
  qui envoie un point~$P$ de la $i$-\`eme copie de~$E$
  sur $2P+P_i$.
\end{rema}
  D'apr\`es
  \cite[\S13.1, th\'eor\`eme 3]{serre:mordellweil},
  dans l'espace projectif, la contribution du nombre de points
  de hauteur born\'ee
  d'un ensemble mince est n\'egligeable. Du point de vue
  du programme de Batyrev et Manin, un ensemble mince qui ne
  v\'erifie pas cela est pathologique. D\'efinissons cela plus
  pr\'ecis\'ement.
\begin{defi}
  Soit $V$~une bonne vari\'et\'e sur~$\KK$ munie d'une
  m\'etrique ad\'elique
  et soit~$T$ une partie mince non vide de $V(\KK)$. On dit que~$T$ est
  \emph{faiblement accumulatrice} si pour tout ouvert
  $U$ de~$V$ pour la topologie de Zariski
  qui rencontre l'adh\'erence de $T$, il existe un
  ouvert de Zariski non vide~$W$ de $V(\KK)$ tel
  que
  \[\mathop{\overline \lim}_{B\to+\infty}
  \frac{\card(T\cap U(\KK))_{H\leq B}}{\card(W(\KK))_{H\leq B}}>0,\]
  la limite sup\'erieure \'etant consid\'er\'ee dans $\RR\cup\{+\infty\}$.
\end{defi}
Les contre-exemples connus \`a la question initiale de Batyrev
et Manin proviennent d'ensembles
minces faiblement accumulateurs qui sont denses pour la topologie
de Zariski.
\subsection{Les surfaces}
Dans le cas des surfaces, les ensembles accumulateurs connus sont
donn\'es par des \emph{courbes exceptionnelles},
que nous d\'efinissons ici comme les courbes rationnelles
lisses d'auto-intersection
n\'egative. Nous parlerons de \emph{belle surface} pour une belle vari\'et\'e
de dimension~$2$.

\begin{prop}
  Soit~$S$ une belle surface sur $\KK$ et soit $L$ une courbe
  exceptionnelle de~$S$. Alors $l(x)=0$ pour tout
  point $x$ de $L(\KK)$ en dehors d'un nombre fini.
\end{prop}
\begin{proof}
  Par la formule d'adjonction, on a la relation
  \[\deg(\omega_L)=L.L+L.\omega_S\]
  o\`u le point d\'esigne le degr\'e d'intersection.
  Comme $L$ est une courbe rationnelle exceptionnelle,
  on en d\'eduit que $L.\omega_S^{-1}<2$. Fixons un
  isomorphisme $\varphi:\PP^1_\KK\to L$.
  On obtient l'in\'egalit\'e
  $\mu_1(\varphi)+\mu_2(\varphi)\leq 2$. Mais l'application $T\varphi$
  fournit un morphisme non nul de $\mathcal O_{\PP^1_\KK}(2)$ dans
  $\varphi^*(TS)$, ce qui prouve que $\mu_1(\varphi)\geq 2$
  et donc $\mu_2(\varphi)<0$. On applique alors
  la proposition~\ref{prop.courbes}.
\end{proof}
\begin{rema}
  La proposition indique que pour toute application~$\varepsilon$
  appartenant \`a~$\mathcal D$,
  le cardinal de l'intersection de l'ensemble
  $S(\KK)_{H\leq B}^{\elibre}$ avec $L(\KK)$ est major\'e par
  une constante ind\'ependante de~$B$.
\end{rema}
\subsection{Les fibrations}
Rappelons que le contre-exemple de V.~V. Batyrev et Y.~Tschinkel
\cite{batyrevtschinkel:counter} repose sur le fait g\'eom\'etrique
que, dans une fibration, le rang du groupe de Picard d'une
fibre peut varier en \'etant \'eventuellement sup\'erieur
au rang du groupe de Picard de la fibre g\'en\'erique,
auquel cas le nombre de points sur chacune de ces fibres peut
faire appara\^\i tre une puissance de $\log(B)$ sup\'erieure
\`a celle attendue pour l'ensemble de la vari\'et\'e.
Nous allons maintenant voir que le fait d'imposer en outre
une minoration sur la libert\'e borne la hauteur des points consid\'er\'es
dans une fibre donn\'ee si bien qu'il n'y a plus de contradiction
entre le nombre de points esper\'e dans une fibre et
celui esp\'er\'e pour la vari\'et\'e.

\begin{prop}
  \label{prop.fibration}%
  Soit~$X$ et~$Y$ de belles vari\'et\'es de dimensions respectives~$m$
  et~$n$ avec $n<m$. Soit $\varphi:X\to Y$
  un morphisme dominant. Soit~$\varepsilon$ une application de $\mathcal D$.
  Il existe une constante~$C$ telle que,
  pour tout point rationnel~$y$ de~$Y$
  qui n'est pas une valeur critique pour~$\varphi$,
  tout point~$x$ de $X_y(\KK)_{H\leq B}^{\elibre}$ v\'erifie
  \[H(x)\leq \min(B,CH(y)^{\frac{m}{n\varepsilon(B)}}).\]
\end{prop}
\begin{proof}
  Cela d\'ecoule de la proposition~\ref{prop.inegalite.morph.var}
  et du fait que la libert\'e de $y$ est major\'ee par~$1$.
\end{proof}
\begin{rema}
  Le minimum est donn\'e par le deuxi\`eme terme d\`es que
  \[H(y)<\left(\frac BC\right)^{\frac{n\varepsilon(B)}m}.\]
\end{rema}
\begin{calc}
  Nous allons maintenant tenter d'expliquer comment cet\-te proposition
  apporte conjecturalement une r\'eponse au contre-exemple de V.~V. Batyrev
  et Y.~Tschinkel. Ce contre-exemple est donn\'e
  par l'hypersurface~$X$ de $\PP^3_\QQ\times\PP^3_\QQ$ d'\'equation
  $\sum_{i=0}^3Y_iX_i^3=0$. Notons $\varphi:V\to\PP^3_\QQ$ la
  seconde projection. Pour tout point ${y=(y_0:y_1:y_2:y_3)}$ de
  $\PP^3(\QQ)$ on d\'esigne
  par $X_y$ la fibre correspondante qui est une surface cubique non
  singuli\`ere si $\prod_{i=0}^3y_i\neq 0$.
  L'application $H:X(\QQ)\to\RR$ d\'efinie
  par
  \[H(x,y)=H_3(x)H_3(y)^3,\]
  o\`u $H_3$ est d\'efinie dans l'exemple~\ref{exem.proj.naif},
  est une hauteur sur~$V$ relative au fibr\'e anticanonique.
  On note $T$ l'ensemble des points $y\in\PP^3(\QQ)$
  en lesquels la fibre $X_y$ est non singuli\`ere et
  a un groupe de Picard de rang strictement sup\'erieur \`a $1$.
  Pour tout $y\in T$, on note $U_y$ le compl\'ementaire
  des $27$ droites de $X_y$. On fera l'\emph{hypoth\`ese}
  qu'il existe une constante $C_y>0$ v\'erifiant
  \[\card U_y(\QQ)_{H\leq B}\leq C_yB\log(B)^5\]
  pour $y\in T$ et $B\geq 2$ et qu'il existe un nombre
  $\delta>0$ tel que \[\sum_{\{y\in T\mid H(y)\geq B\}}C_y\leq B^{-\delta}.\]
  Cette hypoth\`ese est, \`a la connaissance de l'auteur,
  compatible avec le comportement attendu pour le nombre
  de points de hauteur born\'ee sur une surface cubique.
  Choisissons $0<\eta<\frac 35$.
  Sous les hypoth\`eses pr\'ec\'edentes, on obtient les majorations
  \[\begin{aligned}
  &\sum_{y\in T}\card U_y(\QQ)_{H\leq B}^{\elibre}\\
  &\leq\!
  \sum_{\{y\in T\mid H(y)\leq B^{\eta\varepsilon(B)}\}}
  C_y(CH(y)^{\frac5{3\varepsilon (B)}})\log(B)^5
  +\!\!\sum_{\{y\in T\mid H(y)\geq B^{\eta\varepsilon(B)}\}}\!\!C_yB\log(B)^5\\
  &\leq C' B^{\eta\frac53(1+\varepsilon(B))}\log(B)^5+
  C'B^{1-\eta\delta\varepsilon(B)}\log(B)^5,
  \end{aligned}\]
  pour une constante $C'$ convenable. Compte tenu de la
  condition (ii) introduite dans la d\'efinition~\ref{defis.epsilon.libre}
  pour l'ensemble~$\mathcal D$, le terme $B^{1-\eta\varepsilon(B)}\log(B)^5$
  est n\'egligeable devant~$B$. En prenant $\eta<\frac 35$, la condition
  (i) de cette d\'efinition assure que le premier terme est \'egalement
  n\'egligeable devant~$B$. Sous l'hypoth\`ese faite sur les fibres
  dont le groupe de Picard est grand, la condition de minoration
  de la libert\'e rendrait donc bien n\'egligeable la contribution
  de cet ensemble mince.
\end{calc}
\begin{rema}
  Il est important de noter que la libert\'e d'un point ne distingue
  pas les points dans les \og mauvaises\fg\ fibres, c'est-\`a-dire
  celles pour lesquelles le rang du groupe de Picard est strictement plus
  grand que~$1$. Du point de vue de la libert\'e, tout point de grande
  hauteur au-dessus d'un point de petite hauteur est consid\'er\'e comme
  \og mauvais\fg. De prime abord, on peut croire que c'est un
  d\'efaut de cet invariant. N\'eanmoins, l'appartenance \`a une
  fibre dont le groupe de Picard est grand n'est pas stable par
  extension de corps. En fait, de ce point de vue,
  tout point est potentiellement
  mauvais: il suffit de passer \`a une extension qui d\'eploie l'action
  du groupe de Galois sur le groupe de Picard de la fibre; a contrario,
  la libert\'e d'un point rationnel est stable par extension de corps.
\end{rema}
\subsection{Des exemples de C.~Le Rudulier}
Dans sa th\`ese, C.~Le Rudulier a construit de nouveaux contre-exemples
\`a la question initiale de V.~Batyrev et Y.~Manin \cite{rudulier:surfaces}.
Ces exemples sont des espaces de modules de Hilbert pour des surfaces.
\subsubsection{Le cas du produit de droites projectives}
Nous allons commencer par rappeler les d\'etails d'un de ces contre-exemple.
On consid\`ere la vari\'et\'e~$V$ d\'efinie comme le sch\'ema de Hilbert
des points de degr\'e deux sur la surface $S=\PP^1_\QQ\times\PP^1_\QQ$.
On note~$Y$ le produit sym\'etrique
$\Sym^2(S)$. La vari\'et\'e~$Y$
est singuli\`ere le long de l'image~$\Delta$ de la diagonale de $S^2$
et le morphisme de Hilbert-Chow $f:V\to Y$ est
une d\'esingularisation de~$Y$. On dispose \'egalement du morphisme
de projection $g:S^2\to Y$.
On note $Z=f^{-1}(\Delta)$.
L'ensemble $M=f^{-1}(g(S^2(\QQ)))-Z(\QQ)$ est une partie mince
dans~$V$ mais dense pour la topologie de Zariski.
D'autre part, on a un morphisme
\[p_1:\Sym^2(\PP^1_\QQ)\times \PP^1_\QQ\to V\]
provenant de l'application
\[(\PP^1_\QQ)^2\times\PP^1_\QQ\to(\PP^1_\QQ\times\PP^1_\QQ)^2\]
donn\'ee par $((x,y),z)\mapsto ((x,z),(y,z))$.
On obtient \'egalement un morphisme 
\[p_2:\PP^1_\QQ\times\Sym^2(\PP^1_\QQ)\to V\]
par sym\'etrie.
On notera~$Z'$ l'adh\'erence de la r\'eunion des
images de $p_1$ et $p_2$ dans~$V$.
L'ensemble $U_0=V\setminus Z\cup Z'$ est un
ouvert de Zariski non vide de~$V$.
Le th\'eor\`eme~5.1 de \cite{rudulier:surfaces}
contient le r\'esultat suivant:
\begin{theo}[C. Le Rudulier]
  Il existe un nombres r\'eel $c>0$ tel que
  pour tout ouvert non vide~$U$ de~$V$ contenu
  dans~$U_0$ on ait les \'equivalences
  \[\card(U(\QQ)\cap M)_{H\leq B}\sim cB\log(B)^3\]
  et
  \[\card(U(\QQ)\setminus M)_{H\leq B}\sim C(V)B\log^2(B),\]
  quand~$B$ tend vers~$+\infty$.
\end{theo}
\begin{rema}
  Comme le rang du groupe de Picard de~$V$ vaut~$3$,
  le terme de droite de la seconde \'equivalence correspond au
  comportement esp\'er\'e
  pour la version raffin\'ee du principe de Batyrev et Manin.
\end{rema}

Les points de l'ensemble mince~$M$ peuvent \^etres caract\'eris\'es de
la fa\c con suivante:
on consid\`ere le morphisme de vari\'et\'es
\[\Delta:\Sym^2(\PP^1_\QQ)=\PP^2_\QQ\longrightarrow \PP^2_\QQ\]
d\'efinie par $\Delta(a:b:c)=(a^2:b^2-4ac:c^2)$.
On obtient alors par composition un morphisme
qu'on note $\Delta_1$
\[V\longrightarrow\Sym^2(S)\longrightarrow(\Sym^2(\PP^1_\QQ))^2
\stackrel{\Delta\circ\pr_1}{\longrightarrow}\PP^2_\QQ.\]
On d\'efinit de m\^eme $\Delta_2:V\to \PP^2_\QQ$.
On note \'egalement $\square:\PP^2_\QQ\to\PP^2_\QQ$
l'application d\'efinie par ${(u:v:w)\mapsto(u^2:v^2:w^2)}$.
Les \'el\'ements de $M$
sont les \'el\'ements de $V(\QQ)$ dont
les images par $\Delta_1$ et $\Delta_2$ sont dans l'image de l'application
$\square$. Le calcul fait dans le paragraphe pr\'ec\'edent pour les
fibrations s'applique aussi bien \`a $\Delta_1$ qu'\`a $\Delta_2$,
si bien que la libert\'e d'un point d'une fibre fix\'ee de~$\Delta_1$
ou de $\Delta_2$ tend vers $0$.
\begin{rema}
  Cet argument repose de fa\c con cruciale sur le fait
  que le rang du groupe de Picard de $\PP^1_\QQ\times\PP^1_\QQ$ est
  strictement sup\'erieur \`a $1$. Cet argument peut \^etre
  \'etendu aux autres surfaces de Del Pezzo
  dont le groupe de Picard v\'erifie cette condition.
\end{rema}
\subsubsection{Le cas du plan projectif}
Dans ce paragraphe, la lettre~$V$ d\'esigne le sch\'ema de Hibert
des points de de degr\'e deux sur le plan projectif.
On note \'egalement~$Y$ le produit sym\'etrique $\Sym^2(\PP^2_\QQ)$
et~$f$ de~$V$ vers~$Y$ le morphisme de Hilbert-Chow qui est l'\'eclatement
de~$Y$ le long de l'image $\Delta$ de la diagonale. On d\'esigne
par $g$ la projection de $(\PP^2_\QQ)^2$ dans~$Y$
et par $U_0$ le compl\'ementaire de $f^{-1}(\Delta)$ dans~$V$.
Soit $M=f^{-1}(g(\PP^2(\QQ)^2))\cap U_0(\QQ)$.
L'ensemble $M$ est mince mais dense dans~$V$ pour
la topologie de Zariski.
D'apr\`es~\cite{rudulier:surfaces},
on peut choisir une hauteur~$H$ relative \`a $\omega_V^{-1}$ sur~$V$
de sorte qu'on ait la relation
\[H(P)=H_2(x)^3H_2(y)^3\]
pour tous $x,y\in\PP^2(\QQ)$ et tout $P\in V(\QQ)$ tel
que $f(P)=g(x,y)$. Rappelons l'enonc\'e du th\'eor\`eme~3.7
de~\cite{rudulier:surfaces}:
\begin{theo}[C.~Le Rudulier] \textup{a)} Pour tout ouvert~$U$ non vide
  de~$V$, on a l'\'equivalence
  \[\card(U(\QQ)\cap M)_{H\leq B}\sim \frac 8{\zeta(3)^2}B\log(B).\]
  lorsque $B$ tend vers $+\infty$.
  \par
  \textup{b)}
  On a l'\'equivalence
  \[\card(V(\QQ)\setminus M)_{H\leq B}\sim C(V)B\log(B).\]
\end{theo}
Nous allons en d\'emontrer un corollaire:
\begin{corr}\label{corr.rudulier.plan}
\textup{a)} Pour tout ferm\'e strict $F$ de~$V$ on a
  \[\card(F(\QQ)\cap M)_{H\leq B}=o(B\log(B)).\]
  \textup{b)}
  Pour tout ouvert non vide~$U$ de~$V$,
  il existe une hauteur $H$ relative \`a l'oppos\'e du fibr\'e canonique
  sur~$V$ telle que le quotient
  \[\frac{\card U(\QQ)_{H\leq B}}{C_H(V)B\log(B)}\]
  ne converge pas vers $1$ lorsque $B$ tend vers $+\infty$,
  o\`u $C_H(V)$ d\'esigne ici la constante empirique associ\'ee
  \`a la norme ad\'elique choisie sur le fibr\'e $\omega_V^{-1}$.
\end{corr}
\begin{listrems}
  \remarque La premi\`ere assertion signifie que la partie~$M$ n'est pas
  r\'eunion de sous-vari\'et\'es faiblement accumulatrices et ne
  peut donc pas \^etre d\'etect\'ee par une m\'ethode de r\'ecurrence
  sur la dimension.
  \remarque La seconde assertion implique que la formule (2.3.1)
  de \cite{peyre:fano} n'est pas v\'erifi\'ee, bien que la r\'eponse
  \`a la question initiale de Batyrev et Tschinkel soit positive dans ce cas.
\end{listrems}
\begin{proof}
  La premi\`ere assertion du corollaire est un cons\'e\-quen\-ce de l'assertion a)
  du th\'eor\`eme de C\'ecile Le Rudulier.
  Pour l'assertion b) nous allons nous inspirer d'une id\'ee de D. Loughran,
  en nous basant sur le lemme suivant:\noqed
\end{proof}
\begin{lemm}
  Il existe une partie ferm\'ee~$F$ de $V(\Adeles_\QQ)$ qui contient~$M$
  et qui est de mesure nulle pour toute mesure d\'efinie par une
  norme ad\'elique sur~$\omega^{-1}_V$.
\end{lemm}
\begin{proof}
  Le sch\'ema de Hilbert consid\'er\'e est d\'efini sur $\ZZ[\frac 12]$
  avec une bonne r\'eduction en tout nombre premier impair.
  Soit $p$ un nombre premier impair.
  Dans $V(\FF_p)$, il y a trois types de points:
  \begin{assertions}
  \item Ceux sur l'image inverse de $\Delta$;
  \item Ceux correspondant \`a des paires de points distincts
    de $\PP^2(\FF_q)$;
  \item Les points de degr\'e deux dans $\PP^2_{\FF_q}$.
  \end{assertions}
  On note $N_a(p)$, $N_b(p)$ et $N_c(p)$ le cardinal
  des trois ensembles correspondants. Ils sont donn\'es par
  \[N_a(p)=(1+p+p^2)(1+p),\quad
  N_b(p)=\frac{(1+p+p^2)(p+p^2)}2\]
  et
  \[N_c(p)=\frac{p^4-p}2.\]
  En particulier $N_b(p)/\card V(\FF_p)$ converge
  vers $1/2$ lorsque le nombre premier~$p$ tend
  vers $+\infty$. Notons $V_b(p)$, pour $p$ premier impair,
  l'ensemble des points de $V(\QQ_p)$ qui se r\'eduisent modulo~$p$
  en un point de type b). Alors l'ensemble
  \[F=V(\RR)\times V(\QQ_2)\times\prod_{p\not\in \{2,\infty\}}
  V_b(p).\]
  est une partie ferm\'ee de $V(\Adeles_\QQ)$ qui contient~$M$
  et de mesure nulle pour toute mesure induite par une norme ad\'elique
  sur $\omega_V^{-1}$.
\end{proof}
\begin{proof}[Fin de la preuve du corollaire~\ref{corr.rudulier.plan}]
  Soit~$U$ un ouvert non vide de~$V$. On raisonne par l'absurde
  en supposant que pour tout choix de norme le quotient converge vers~$1$.
  Par la d\'emonstration de l'assertion (c)
  de la  proposition (3.3) de \cite{peyre:fano},
  les points de $U(\QQ)$ v\'erifie la propri\'et\'e ($E_U$) de
  \cite[\S3]{peyre:fano}. En appliquant cela au compl\'ementaire
  du ferm\'e~$F$ donn\'e par le lemme pr\'ec\'edent, il en r\'esulte
  que le quotient
  \[\frac{\card(U(\QQ)\cap M)_{H\leq B}}{\card U(\QQ)_{H\leq B}}\]
  tend vers~$0$ quand $B$ tend vers~$+\infty$, ce qui contredit
  l'assertion a) du th\'eor\`eme de C.~Le Rudulier.
\end{proof}

En ce qui concerne les pentes, contrairement au cas pr\'ec\'edent,
on ne dispose pas d'un morphisme de vari\'et\'es $\varphi$ qui permette
d'utiliser la proposition~\ref{prop.fibration} et qui v\'erifie
en outre $\varphi(M)\neq\varphi(V(\QQ))$. Toutefois, l'application
qui \`a une paire de points distincts $\{P_1,P_2\}$ associe la droite
projective $(P_1P_2)$ d\'efinit un morphisme surjectif de vari\'et\'es
$V\to\PP^2_\QQ$ dont la fibre au-dessus d'une droite projective~$D$
est isomorphe \`a $\Hilb^2(D)$, c'est-\`a-dire \`a un plan projectif.
Compte tenu de la proposition~\ref{prop.fibration},
il existe une constante~$C$ de sorte que,
pour tout point~$x$ de~$M$ donn\'e par une paire $\{P_1,P_2\}$
contenue dans une droite $D$ du plan projectif,
la condition $l(x)>\varepsilon(B)$ implique
\[H(P_1)H(P_2)\leq CH(D)^{\frac 2{\varepsilon(B)}}.\]
Cette condition ne d\'ecoule pas d'une condition
de la forme $l(P_1,P_2)\geq \varepsilon'(B)$ o\`u $(P_1,P_2)$
est vu comme un point de $\PP^2(\QQ)^2$. Il n'y a donc plus
de contradiction directe entre le r\'esultat
connu pour $(\PP^2_\QQ)^2$
et celui esp\'er\'e pour $V$. Toutefois je n'ai pas r\'eussi
\`a d\'emontrer que le cardinal  $M_{H\leq B}^{\elibre}$ est
effectivement de la forme $o(B\log(B))$ ce qui prouverait que les pentes
permettent d'\'ecarter assez de mauvais points
dans ce cas particulier. Une \'etude plus approfondie de ce
cas m\'eriterait d'\^etre faite.
\section{Conclusion}
\`A l'issu de ce travail, l'auteur est persuad\'e que les
pentes de l'espace tangent donnent effectivement un indicateur
fid\`ele permettant de d\'etecter les points qui s'accumulent.
Toutefois il est encore difficile de pr\'edire avec
certitude quelle condition pr\'ecise se r\'ev\`elera la plus
efficace pour poursuivre le programme de Batyrev et Manin.
Deux crit\`eres de nature diff\'erente apparaissent ici:
la condition $l(x)>\epsilon(B)$, qui vient s'ajouter
\`a la condition de hauteur, et la condition
$\mu_{\max}(x)\leq\log(B)$.
Si le second point de vue am\`ene un changement de paradigme plus profond,
il peut \^etre plus naturel. Une troisi\`eme solution
consisterait \`a imposer une minoration sur la pente minimale
de la forme $\mu_{\min}(x)\geq\eta\log(\log(B))$. Cette
solution qui peut \^etre plus facile \`a v\'erifier
dans certains cas a l'inconv\'enient que l'ensemble des
\og bons\fg\ points diminue avec~$B$, elle n'a donc pas \'et\'e retenue
par l'auteur pour des raisons m\'etamath\'ematiques.
Seule l'\'etude d'autres cas permettra de trancher d\'efinitivement.
\let\bold\mathbf
\def\andname{et }
\def\comma{}
\ifx\undefined\bysame
\newcommand{\bysame}{\leavevmode\hbox to3em{\hrulefill}\,}
\fi
\ifx\undefined\numero
\newcommand{\numero}{$\hbox{n}^\circ$}
\fi
\ifx\undefined\andname
\newcommand{\andname}{and }
\fi
\ifx\undefined\comma
\newcommand{\comma}{,}
\fi

\end{document}